\documentclass[twocolumn,hidelinks,hidempi]{mpi2015-cscpreprint}

\usepackage[utf8]{inputenc}
\usepackage{mymacros}

\usepackage{orcidlink}
\usepackage{fontawesome}
\usepackage{amsthm}
\AtBeginDocument{%
\newtheorem{proposition}{Proposition}[section]
}

\begin{document}

\title{Towards an Efficient Shifted Cholesky~QR\\for Applications in Model Order Reduction using pyMOR}
\shorttitle{Towards Efficient Shifted Cholesky~QR for {\pymor}}
\shortauthor{M.~Bindhak, A.~J.~R.~Pelling, J.~Saak}
\author[$\ast$, \faEnvelopeO]{Maximilian Bindhak~\orcidlink{0009-0005-4627-9574}}
\author[$\dagger$]{Art~J.~R.~Pelling~\orcidlink{0000-0003-3228-6069}}
\author[$\ast$]{Jens Saak~\orcidlink{0000-0001-5567-9637}}

\affil[$\ast$]{%
  Computational Methods in Systems and Control Theory,\authorcr%
  Max Planck Institute for Dynamics of Complex Technical Systems,\authorcr%
  Magdeburg, Germany}
\affil[$\dagger$]{
  Department of Engineering Acoustics,\authorcr%
  Technische Universit\"at Berlin, Germany}
\affil[\faEnvelopeO]{Corresponding author, \protect\email{bindhak@mpi-magdeburg.mpg.de}}

\abstract{%
    Many model order reduction (MOR) methods rely on the computation of an orthonormal
    basis of a subspace onto which the large full order model is
    projected. Numerically, this entails the orthogonalization of a set
    of vectors. The nature of the MOR process imposes several requirements
    for the orthogonalization process. Firstly, MOR is oftentimes
    performed in an adaptive or iterative manner, where the quality of the
    reduced order model, i.e., the dimension of the reduced subspace, is
    decided on the fly. Therefore, it is important that the
    orthogonalization routine can be executed iteratively. Secondly, one
    possibly has to deal with high-dimensional arrays of abstract vectors that
    do not allow explicit access to entries, making it
    difficult to employ so-called `orthogonal triangularization
    algorithms' such as Householder~QR.

    For these reasons, (modified) Gram-Schmidt-type algorithms are
    commonly used in MOR applications. These methods belong to the category of
    `triangular orthogonalization' algorithms that do not rely on
    elementwise access to the vectors and can be easily
    updated. Recently, algorithms like shifted Cholesky~QR have gained
    attention. These also belong to the aforementioned category and have proven
    their aptitude for MOR algorithms in previous studies. A key benefit
    of these methods is that they are communication-avoiding, leading to
    vastly superior performance on memory-bandwidth-limited problems and
    parallel or distributed architectures. This work formulates an
    efficient updating scheme for Cholesky~QR algorithms and proposes
    an improved shifting strategy for highly ill-conditioned matrices.

The proposed algorithmic extensions are validated with numerical experiments on a laptop and computation server.
}
\novelty{Our contribution in this work is two-fold.
  \begin{itemize}
  \item We propose some algorithmic changes to the shifted
    Cholesky~QR algorithm which imrpove its numerical robustness and
    performance in certain contexts.
  \item We execute a thorough
    performance comparison in a situation where only vector-wise access to the
    matrix columns is allowed.
  \end{itemize}
}
\maketitle 

\section{Introduction}%
\label{sec:introduction}
Mathematical modeling of complex dynamical systems oftentimes produces very large models that can become numerically infeasible in practice. The endeavor of finding smaller, and therefore more efficient, surrogate models while retaining the most important dynamics of the system is known as model order reduction (MOR)~\cite{antoulas2005,benner2017,morBenGQetal21,morBenGQetal21a,morBenGQetal21b}. The majority of MOR methods achieve the reduction of model size by means of projection. Hereby, the degrees of freedom (DOFs) of the full order model (FOM) are projected onto a subspace of smaller dimension. The main challenge in MOR is finding a subspace that is small, yet leads to accurate reduced order models (ROMs). These subspaces are usually represented by bases of orthonormal vectors that are constructed during the MOR process. Therefore, orthogonalization routines are required in many MOR tasks such as moment matching (MM)~\cite{morGri97,morFre03}, the iterative restarted Krylov algorithm (IRKA)~\cite{gugercin2008}, proper orthogonal decomposition (POD)~\cite{morSir87} and also recent data-driven MOR methods such as the randomized Eigensystem Realization algorithm (rERA)~\cite{minster2021}.

The landscape of orthogonalization methods is vast and growing with many recent developments~\cite{Fuk19,CarLR21,carson2022,CarLMetal25}. Following the conceptualization in~\cite{trefethen1997,fukaya2020}, orthogonalization algorithms can be divided into so-called \emph{orthogonal triangularization} algorithms that successively left-multiply orthogonal matrices until a triangular matrix is obtained, and \emph{triangular orthogonalization} algorithms that multiply triangular matrices from the right until an orthogonal matrix is obtained. The former \emph{orthogonal triangularization} methods include all Householder-type QR algorithms~\cite{bischof1987,schreiber1989}, whereas Gram-Schmidt-type methods, and recently also Cholesky~QR-type~\cite{fukaya2014,fukaya2020,fukaya2024} algorithms, belong to the latter class of \emph{triangular orthogonalization} methods.

Even though there are plenty of orthogonalization methods to choose from, the application at hand imposes constraints on the choice of method. In this work, we will look through the lens of MOR applications, specifically using the free and open source Python software library \pymor{}~\cite{milk2016,pymor}. A special requirement of the orthogonalization methods in \pymor{} is the abstract operator and \texttt{VectorArray} interface. The key idea is that the MOR process is formulated in an abstract manner whereby the construction and solution of FOMs is offloaded to established software libraries.

On the one hand, this abstraction enables easy integration of almost any arbitrary solver for partial differential equations (PDEs) because boilerplate code is largely avoided, increasing the applicability of the algorithms. Further, it allows MOR routines to be implemented such that they are independent of the concrete PDE problem and solver. Thus, the maintenance effort for the MOR code is reduced drastically and the integration of new algorithms by non-experts regarding the specific PDE solvers is straightforward. This gives maximum flexibility to try the methods on a whole range of problems.
On the other hand, tall and skinny matrices become mere arrays of vectors that reside in the PDE solver's memory space and their concrete realization may be such that access to single degrees of freedom (DOFs) is either very expensive or even impossible. This is especially true when the PDE solver is run in distributed parallel mode and significant bookkeeping is required to identify on which of the processes, i.e., involved hosts, the required DOFs reside.

In effect, this introduces several numerical and algorithmic restrictions when dealing with large-scale problems: As the software interface has to accommodate a wide range of backends, only entire FOM solutions can be queried and single DOFs generally cannot be accessed. In consequence, \emph{orthogonal triangularization} algorithms such as Householder-type methods can mostly be ruled out for MOR applications in \pymor{}, as DOF-access is needed to construct the Householder reflectors. Also, the communication-avoiding so-called tall and skinny QR decomposition~(TS-QR)~\cite{DemGHetal12}\footnote{An extensive comparison between TS-QR with the repeated Cholesky QR was done in~\cite{fukaya2014}, which shows significant speedups for the latter on parallel distributed systems.}, which splits the tall and skinny matrix into blocks vertically, requires DOF-access to achieve this splitting, hence cannot be implemented in terms of the present abstraction. This leaves the class of \emph{triangular orthogonalization} methods such as Gram-Schmidt type algorithms and the Cholesky~QR variants which will be considered in this work.

In a nutshell, the basic Cholesky~QR algorithm relies on the observation that
the QR decomposition of a matrix \(A\) can be computed from a decomposition of
its Gramian \(X=A\tra A\). Namely, from a Cholesky factor \(R\) of the Gramian \(X=R\tra R\), where \(R\) is upper triangular, the QR decomposition \(A=QR\) can be obtained by \(Q=AR\inverse\). Of course, this simple relationship only holds in exact arithmetic and this basic algorithm is unstable in practice. Several algorithmic extensions \cite{fukaya2014,fukaya2020,fukaya2024} of the basic algorithm have been proposed recently that solve the issue of numerical instability.

For QR decompositions in MOR applications, we can generally distinguish two modes of operation:
\begin{enumerate}[{M\,}1]
\item In algorithms like IRKA~\cite{gugercin2008}, or when truncating the rank of the system Gramians in balancing-based MOR~\cite{Saa09}, one obtains the entire non-orthonormal bases and run a full QR sweep all-at-once. Here, comparably many columns at a time are being orthonormalized and approaches based on Cholesky~QR are well-suited.\label{mode_one}
\item In adaptive approaches, such as MM with a posteriori error estimates~\cite{morChe22}, adaptive data-driven modeling~\cite{pelling2025}, or in POD-Greedy approaches, like the reduced basis (RB) method~\cite{morBenGQetal21a}, we successively increase the number of basis vectors. That means, one extends an existing orthonormal basis by one or more new vectors that still need orthonormalization. In other words, one performs a QR update. Especially in POD/RB-like approaches, it then often happens that less and less of the new vectors are actually sufficiently linearly independent of the existing basis, i.e., the condition number of the full matrix of basis vectors becomes increasingly large. Usually, Gram-Schmidt-type approaches are then used, as they allow detection of the linear dependence and vectors can be dropped accordingly. However, this detection may only happen after comparison to many previously computed orthonormal vectors, causing significant communication (memory movements) overhead. In this work, we argue that Cholesky~QR-type algorithms can be computationally advantageous in this mode of operation as well.\label{mode_two}
\end{enumerate}

In both cases, the classic Cholesky~QR algorithm~\cite[Theorem 5.2.3]{GolV13} suffers from a potential linear dependence of basis vectors, as the computation of the Gramian squares the conditioning of the problem and very ill-conditioned matrices can enter the Cholesky decomposition. The shifted Cholesky~QR~\cite{fukaya2014,fukaya2020} is a remedy to this issue, but it can become slow when the shift needs to be repeated very often or, as demonstrated later, might outright fail. Moreover, in \pymor's abstraction, the Gramian computation can be executed as a BLAS level-3-type GEMM operation only for \texttt{NumpyVectorArrays}\footnote{\url{https://docs.pymor.org/2024-2-0/autoapi/pymor/vectorarrays/numpy/index.html}} and is limited to a single inner product of two vectors, i.e., strictly BLAS level-1 per entry of the Gramian, for generic \texttt{VectorArray}s such as \texttt{ListVectorArray}\footnote{\url{https://docs.pymor.org/2024-2-0/autoapi/pymor/vectorarrays/list/index.html}} and others.

In this contribution, we tackle both issues. To this end, we suggest three algorithmic extensions in~\Cref{sec:extensions}. An evaluation of the required number of floating point operations of the corresponding algorithms can be found in \Cref{cha:flopCount}. \Cref{sec:benchmarks} presents an extensive performance evaluation benchmarking the algorithms with respect to execution time and several accuracy measures. This illustrates the numerical robustness of the newly proposed algorithmic extensions for a variety of input matrices.
In \Cref{cha:outlook}, we conclude the paper and discuss possible further extensions and investigations.
%

\section{Extensions of the Shifted Cholesky~QR Algorithm}%
\label{sec:extensions}
Above, we have identified two main issues with the shifted Cholesky~QR. In this
section, we provide a remedy for the potential breakdowns or expensive inner
loops in the presence of linearly dependent vectors, due to ill-conditioned or very large
Gramians $X=A\tra A$. We address the ill-conditioning by adapting the shift in \Cref{sec:recomp}.

As mode M\,\ref{mode_two} mostly relies on QR updates, we propose an efficient
updating scheme for ``Cholesky~QR''-type decompositions. This is finally used in
\Cref{sec:panel_scheme} to devise a panel scheme that successively computes
updated QR decompositions with a block of columns added in each step. This
improves the operations count and the memory locality of the procedure when
using {\pymor}'s \texttt{VectorArray} abstraction and ultimately limits the
dimensions of the Gramian.
\subsection{Shift recomputation}%
\label{sec:recomp}%
As is supported by the numerical experiments in \Cref{cha:bench_chol,cha:bench_errors}, we now propose a more numerically robust variation of the \emph{`Iterated Cholesky QR for
X = QR with shifts when necessary'} algorithm presented in~\cite[Algorithm 4.1]{fukaya2020}\footnote{Note that the naming scheme for $X$ and $A$ in~\cite{fukaya2020} is exactly opposite to the scheme used in this paper.} (\isCholQR).
Experiments have shown that, in very ill-conditioned cases, the {\isCholQR} does not converge to an orthogonal $Q$ because the initial shift estimation based on \(\norm{A}^2\) is too large by orders of magnitude. This large shift is then reapplied in successive iterations, ultimately producing a non-orthogonal \(Q\).
A straightforward solution to this problem is to simply recompute the shift according to the current approximation.
To this end, we present the  \emph{`Iterated Cholesky~QR with shift recomputation'}
(\rsCholQR) in \Cref{algo:cholqr} that, rather than reapplying the initial shift, recomputes the shift whenever the Cholesky decomposition of the current Gramian iterate \(X\) fails. Since $\norm{X}=\norm{A}^2$, we propose to compute the shift based on \(\norm{X}\). Not only does this solve the problem of a non-orthogonal \(Q\) because the shift matches the current iterate, it can further be numerically more efficient: Despite performing shift computations repeatedly instead of only once, multiple computations of \(\norm{X}\) can be less expensive than \(\norm{A}^2\), depending on the dimension of \(A\) and number of required iterations. Furthermore, one can utilize symmetric eigenvalue solvers, like the Lanzcos method~\cite{lanczos50} or implicitly restarted Lanczos method (IRLM)~\cite{sorensen97}, to compute only the largest eigenvalue of the Gramian \(X\) because it is symmetric/hermitian.
{
    \RenewDocumentCommand{\VA}{}{\ensuremath{A}}           
    \RenewDocumentCommand{\VQ}{}{\ensuremath{Q}}           
    \RenewDocumentCommand{\VR}{}{\ensuremath{R}}           
    \RenewDocumentCommand{\VRup}{}{\ensuremath{\tilde{R}}} 
    \begin{algorithm}[th]
    \caption{(\rsCholQR): Iterated Cholesky QR with shift recomputation.\label{algo:cholqr}}
        \begin{algorithmic}[1]
            \Require{$\VA\in\R[m][n]$, unit roundoff $\uro>0$.}
            \Ensure{$\VQ\in\R[m][n],\VR\in\R[n][n]$}
            \State $\VQ \gets \VA$
            \State $\VR \gets I_n$
            \State $\VX \gets \VQ\tra \VQ$
            \While {$\fnorm{\smash{\VX-I_n}}\geq\uro\sqrt{n}$}
                \If {$\VRup \gets \chol(\VX)$ breaks down}
                    \State $\sigma \gets \max\left\{11(mn+n(n+1))\uro\norm{\VX},2\uro\right\}$\label{alg:cholqr:six} \Comment{Recompute shift.}
                    \State $\VRup \gets \chol(\VX + \sigma I_n)$
                \EndIf
            \State $\VQ \gets \VQ\VRup\inverse$
            \State $\VR \gets \VRup\VR$
            \State $\VX \gets \VQ\tra\VQ$
            \EndWhile
        \end{algorithmic}
    \end{algorithm}
}

It is important to note that, in certain instances, it is possible that the shift becomes insignificantly small in later iterations and the proposed algorithm ends up in a quasi-infinite loop. As a remedy, it is advised to limit the magnitude of the shift by twice the unit round-off from below (Line~\ref{alg:cholqr:six} in \Cref{algo:cholqr}).

In order to get an impression of possible savings from this extension, as well as the extensions introduced in the following subsections, we collect the required floating
point operations (flops) in \Cref{tab:rsCholQRFlop}, found in \Cref{cha:flopCount}. Note that especially the \gemm{}\footnote{Here, {\gemm} refers to a general matrix-matrix product from {\blas}, which is usually associated with a highly parallelizable operation. More operations in {\blas}/{\lapack} notation used in this paper can be found in \Cref{tab:routinesFlop}.} operations for forming the Gramians are not memory traffic optimal due to the {\pymor} abstraction, such
that direct comparison with NumPy's \ndarray{} will favor the latter due to its better memory management.
\subsection{Updating the decomposition}%
\label{sec:cholqrupdate}
As mentioned in \Cref{sec:introduction}, many adaptive MOR tasks require a
repeated orthonormalization to extend an existing orthonormal basis with new
vectors (operation mode M\,\ref{mode_two}). To this end, we formulate a
QR-updating scheme for the (shifted) Cholesky~QR algorithm in
\Cref{algo:cholqrupdate}. The core algorithm was recently presented for a
data-driven MOR application by the second author in~\cite[Appendix
A]{pelling2025}. In this work, it is analyzed indirectly via the panel variant introduced in \Cref{sec:panel_scheme} with respect to its numerical
behavior and careful benchmarking against the existing approaches applicable in {\pymor} is undertaken.

Suppose we have an existing QR decomposition \[\VQin\VRin=\VAprev\in\R[m][q]\] and
are looking to obtain a QR decomposition for an updated matrix
$\VA=\begin{bmatrix}\VAprev & \VAin\end{bmatrix}$, i.e., find \(\VQ,\VB,\VR\) such that \[\VQout=\begin{bmatrix}\VQin & \VQ\end{bmatrix}\in\R[m][(q+p)]\] and \[\VRout=\begin{bmatrix}\VRin & \VB\\0&\VR\end{bmatrix}\in\R[(q+p)][(q+p)].\] $\VAin\in\R[m][p]$ denotes the new columns that have been appended to
$\VAprev$, where \(\VQ\) is its orthogonalization with respect to \(\VQin\). The following proposition reveals how an updated QR decomposition can be obtained from a Cholesky approach:
\begin{proposition}
    Let $\VQin\in\R[m][q]$, $\VRin\in\R[q][q]$ be a QR decomposition $\VQin\VRin=\VAprev\in\R[m][q]$ and let $\VR$ be given by the Cholesky factorization
    \begin{align*}
         \VR\tra \VR&=(\VAin -\VQin\VQin\tra \VAin)\tra(\VAin-\VQin\VQin\tra\VAin)\\
         &=\VAin\tra \VAin-\VAin\tra \VQin\VQin\tra \VAin = X
    \end{align*} for some $\VAin\in\R[m][p]$. Then,
    \begin{align*}
        \VQout&=\begin{bmatrix}
            \VQin&(\VAin-\VQin\VQin\tra \VAin)\VR\inverse
        \end{bmatrix}\\\intertext{and}
        \VRout&=
        \begin{bmatrix}
            \VRin&\VQin\tra \VAin\\0&\VR
        \end{bmatrix}
    \end{align*}
    is a QR decomposition of $\VA=\begin{bmatrix}\VAprev & \VAin\end{bmatrix}\in\R[m][(q+p)]$.
\end{proposition}
\begin{proof}
Firstly, we have
{\footnotesize%
\begin{align*}
  \VQout\VRout&=\begin{bmatrix}
    \VQin\VRin & \VQin\VQin\tra \VAin + (\VAin -\VQin\VQin\tra \VAin)\VR\inverse\VR
  \end{bmatrix}\\&=\begin{bmatrix}
    \VAprev&\VAin
  \end{bmatrix}=\VA.
\end{align*}
}%
Secondly, $\VRout$ is upper triangular by construction.

It remains to verify that $\VQout$ is orthogonal. Let \[\VQ = (\VAin-\VQin\VQin\tra \VAin)\VR\inverse,\] such that \[\VQout=\begin{bmatrix}\VQin&(\VAin-\VQin\VQin\tra \VAin)\VR\inverse\end{bmatrix}.\] Then,
\begin{align*}
  \VQin\tra\VQ &= \VQin\tra(\VAin-\VQin\VQin\tra \VAin)\VR\inverse\\
  &= (\VQin\tra \VAin - \VQin\tra \VAin)\VR\inverse = 0,
\end{align*}
and
\begin{align*}
    \hspace*{-4em}\VQ\tra\VQ &= \VR\trainv(\VAin- \VQin\VQin\tra \VAin)\tra (\VAin-\VQin\VQin\tra \VAin)\VR\inverse\\
    &= \VR\trainv(\VR\tra\VR)\VR\inverse = I_{p}.
\end{align*}
Thus, \(\VQout\tra\VQout=I_{q+p}\).
\end{proof}
The computation of the Cholesky factor $\VR$ can be stabilized by applying shifts whenever it breaks down, just as it is done in \Cref{algo:cholqr}.
Multiple iterations might be necessary to reach the desired accuracy.
However, also breakdowns can happen, although they are lucky breakdowns, as we
will explain in the following. To this end, we consider two columns \(a_{i}\) and
\(c_{j}\) of \(\VAin\) with \(a_{i}\) being a linear combination of vectors in
\(\VQin\), i.e., \(a_{i}=\VQin b_{i}\) for some \(b_{i}\in\R[q]\) and \(c_{j}\)
orthogonal on \(\VQin\), i.e., \(\VQ\tra c_{j}=0\). Then, we have
\begin{align*}
  X_{i,i} &= a_i\tra a_i - a_i\tra \VQin \VQin\tra a_i\\
  &= b_i\tra \VQin\tra \VQin b_i - b_i\tra b_i = 0.
\end{align*}
Clearly, this leads to a rank deficiency in the Gramian (as Sylvester's criterion is
violated, and definiteness is lost) and the Cholesky
decomposition will fail. However, in MOR, one is usually interested in a rank
revealing QR decomposition and one could simply proceed with the reduced Gramian
formed without \(a_{i}\).
On the other hand,
\begin{align*}
  X_{j,j} &= c_{j}\tra c_{j} - \underset{=0}{\underbrace{c_{j}\tra\VQin\VQin\tra c_{j}}}\\
         &=  c_{j}\tra c_{j}.
\end{align*}

In practice, obviously anything in between can happen and \(X\) can be
arbitrarily ill-conditioned. To avoid issues with the Cholesky decomposition of \(\VX\), we resort to shifting in this contribution. However, a suitable combination of
shifting and a column pivoting strategy along the lines of~\cite{fukaya2024}
could potentially yield even better results.

The considerations made above also
apply to the panel scheme devised in the next section when successively updating
existing QR decompositions. Especially there, the initial shift is often by
magnitudes to small. Our current workaround is to increase it by a factor of \(10\)
until the shifted Cholesky decomposition succeeds. We have experimented with
different strategies, but have not found any that has similar robustness at lower
computational costs, yet. In any case, we have observed runtime improvements
even with this rather brute force approach, as reported in
\Cref{sec:benchmarks} and~\cite{pelling2025}.

\begin{algorithm}[th]
    \caption{(\cholQRUpdate):\label{algo:cholqrupdate}}
    \begin{algorithmic}[1]
        \Require{Existing QR decomposition $\VQin\in\R[m][q]$, $\VRin\in\R[q][q]$, new columns
          $\VAin\in\R[m][p]$, unit roundoff $\uro>0$.}
        \Ensure{$\VQout\in\R[m][n]$, $\VRout\in\R[n][n]$ that satisfies
          $q+p=n$ and $\VQout\VRout=\begin{bmatrix}\VQin\VRin & \VAin\end{bmatrix} $.}
        \State $\VQ \gets \VAin$
        \State $\VB \gets 0\in\R[q][p]$
        \State $\VR \gets I_p$
        \State $\VBup \gets \VQin\tra\VQ$
        \State $\VX \gets \VQ\tra\VQ - \VBup\tra\VBup$
        \While {$\fnorm{\smash{\VX - I_p}} \geq \uro\sqrt{p}$}
            \State $\sigma \gets 0$
            \While {$\VRup \gets \chol(\VX + \sigma\cdot I_p)$ breaks down}
              \State $\sigma \gets \begin{cases}11(mq+q(q+1))\uro\norm{\VX}&\sigma=0\\ 10\cdot\sigma&\textrm{else}\end{cases}$
              \State $\sigma \gets \max\left\{\sigma,2\uro\right\}$
            \EndWhile
            \State $\VQ \gets (\VQ - \VQin\VBup)\VRup\inverse$
            \State $\VB \gets \VB + \VBup\VR$
            \State $\VR \gets \VRup\VR$
            \State $\VBup \gets \VQin\tra\VQ$ \Comment{Update $\VBup$ for next iteration.}
            \State $\VX \gets \VQ\tra\VQ - \VBup\tra\VBup$ \Comment{Update $\VX$ for next iteration.}
        \EndWhile
        \State $\VQout \gets\begin{bmatrix}\VQin & \VQ\end{bmatrix}$
        \State $\VRout \gets \begin{bmatrix} \VRin & \VB\\ 0 & \VR \end{bmatrix}$
    \end{algorithmic}
\end{algorithm}

\subsection{A panel scheme}%
\label{sec:panel_scheme}%
{
    \RenewDocumentCommand{\VA}{}{\ensuremath{A}}           
    \RenewDocumentCommand{\VQ}{}{\ensuremath{Q}}           
    \RenewDocumentCommand{\VR}{}{\ensuremath{R}}           
    \RenewDocumentCommand{\VRup}{}{\ensuremath{\tilde{R}}} 
    \begin{algorithm}[th]
        \caption{(\pnCholQR): Cholesky~QR panel scheme.}\label{alg:panel_cqr}
        \begin{algorithmic}[1]
            \Require $\VA=\begin{bmatrix}\VA_1 & \VA_2 & \ldots & \VA_r \end{bmatrix} \in\R[m][n]$, unit roundoff $\uro>0$
            \Ensure $\VQ\in\R[m][n], \VR\in\R[n][n]$
            \State $\VQ \gets \begin{bmatrix}\phantom{0}\end{bmatrix}, \VR \gets \begin{bmatrix}\phantom{0}\end{bmatrix}$
            \For{$i=1,2,\ldots,r$}
                \State $\VQ, \VR \gets$ {\cholQRUpdate}$(\VQ, \VR, \VA_i, \uro)$
            \EndFor
        \end{algorithmic}
    \end{algorithm}
}

We introduce a straightforward QR panel scheme based on {\cholQRUpdate} from
\Cref{algo:cholqrupdate}. The main goal of this approach are runtime improvements. The idea is to split the input matrix $A$ into $r$ (ideally) equally sized panels $A=\begin{bmatrix}A_1 & A_2 & \ldots & A_r\end{bmatrix}$ and apply {\cholQRUpdate} iteratively onto these panels. This scheme reduces the size of the Gramian per iteration, allowing for a cheaper Cholesky decomposition and shift evaluation, especially if $X$ is small enough to fit into the lower cache levels. In \Cref{cha:rs_pn_flop_comp}, we analyze the operation counts and show a theoretical decrease in flops by a factor of up to two for an increasing amount of panels compared to {\rsCholQR}. This can potentially also lead to a reduction in power consumption. The pseudocode for this procedure can be found in \Cref{alg:panel_cqr}.

However, the reduction of flops does not necessarily result in a decrease in runtime, as will
be illustrated in \Cref{cha:bench_runtimes}. Unlike, for example, the level-3-type {\lapack} implementations of the LU (\getrf) and QR decomposition ({\geqrf} + {\gqr}), which replace several {\blas} level-1-type and level-2-type calls with level-3-type routines on the full matrix, {\pnCholQR} only introduces more level-3-type
calls with smaller Gramian matrices depending on the size of the panels.
However, the aforementioned issues with {\cholQRUpdate} persist as the algorithm cannot act cleverly in the large direction (\(m\)). In other words, we admittedly cannot avoid the large direction, but we can control the size of the Gramians and the data traffic required for forming them and computing with them.
At the same time, lower level caches get increasingly exclusive to single cores
or core groups, limiting the parallelism in the computations. The latter being
especially harmful for systems with many cores, as we will see in \Cref{sec:benchmarks}. Moreover, we observed fewer repetitions are required to evaluate the QR decomposition of the first few panels, but it becomes successively more expensive for the later panels due to the higher probability/fraction of linear dependence.
%

\section{Benchmarks}%
\label{sec:benchmarks}
\begin{table*}[tp]
  \rowcolors{2}{white}{gray!25!white}
    \centering
    \begin{tabular}{| l | l | l |}\hline
                                                    & Server                    & Notebook              \\\hline
        Operating system                            & Ubuntu 20.04.6 LTS        & Ubuntu 24.04.2 LTS    \\
        Conda version                               & 25.3.0                    & 24.5.0                \\\hline
        Main memory size                            & $\approx$ 1 TB            & 16 GB                 \\\hline
        CPU                                         & 2x {\PCepyc}              & {\PCryzen}            \\
        CPU architecture                            & {\zen} 3                  & {\zen} 4              \\
        Number physical cores (logical)             & 64 (128)                  & 8 (16)                \\\hline
        L3 cache per 8 cores (CPU total)            & 32 MB (256 MB)            & 16 MB (16 MB)         \\
        L2 cache per core (CPU total)               & 512 KB (32.768 MB)        & 512 KB (4.096 MB)     \\
        L1 instructions cache  per core (CPU total) & 32 KB (2.048 MB)          & 32 KB (0.256 MB)      \\
        L1 data cache per core (CPU total)          & 32 KB (2.048 MB)          & 32 KB (0.256 MB)      \\\hline
    \end{tabular}
    \caption{Overview of the operating systems and hardware for the two systems
      used for benchmarking. Server has two CPU sockets, but only the first CPU
      was used ($\approx$ \(500~\)GB RAM).}%
    \label{tab:bench_systems}
\end{table*}
\begin{table*}[tp]
    \rowcolors{2}{white}{gray!25!white}
    \centering
    \begin{tabular}{| l | l | l |}\hline
                                                    & Main Environment  & \fenics{} Environment \\\hline
        Package                                     & Version           & Version               \\\hline
        \openblas~\cite{openblas}                   & 0.3.21            & 0.3.29                \\
        Python                                      & 3.10.17           & 3.12.10               \\
        \numpy~\cite{harris2020array}               & 1.24.4            & 1.26.4                \\
        \scipy~\cite{2020SciPy-NMeth}               & 1.10.1            & 1.15.2                \\
        \pymor~\cite{pymor24.2}                     & 2024.2.0          & 2024.2.0              \\
        \ngsolve~\cite{ngsolve,schoberl1997}        & 6.2.2501          & n.a.                  \\
        \fenics{}~\cite{AlnaesEtal2015,LoggEtal2012}& n.a.              & 2019.1.0              \\\hline
    \end{tabular}
    \caption{Software packages used are installed via conda. Only the most
      important software package versions are listed in this table. For all
      installed packages and their respective versions, please refer to the provided YAML file.}%
    \label{tab:bench_software}
\end{table*}
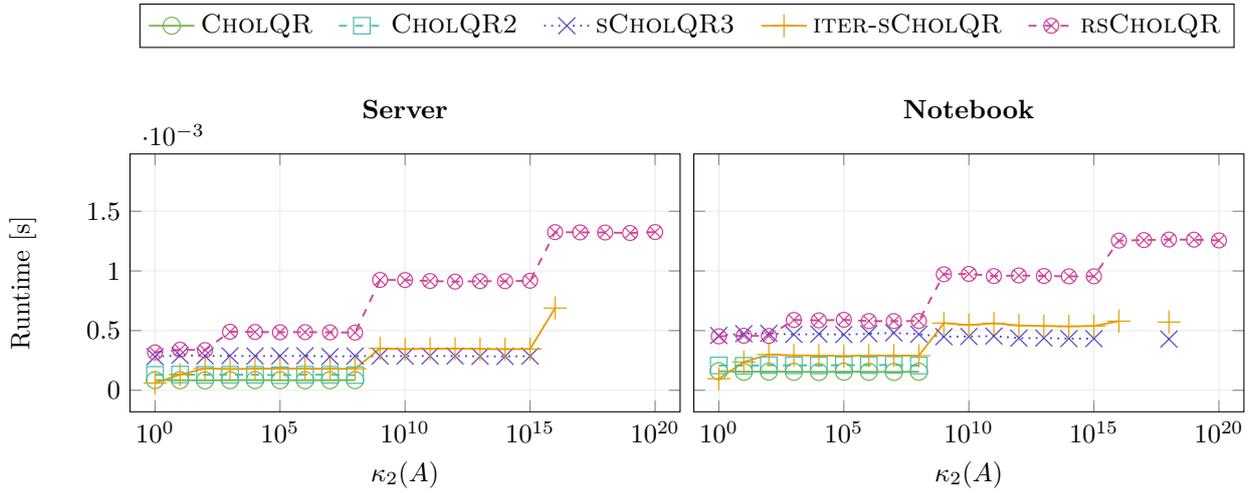
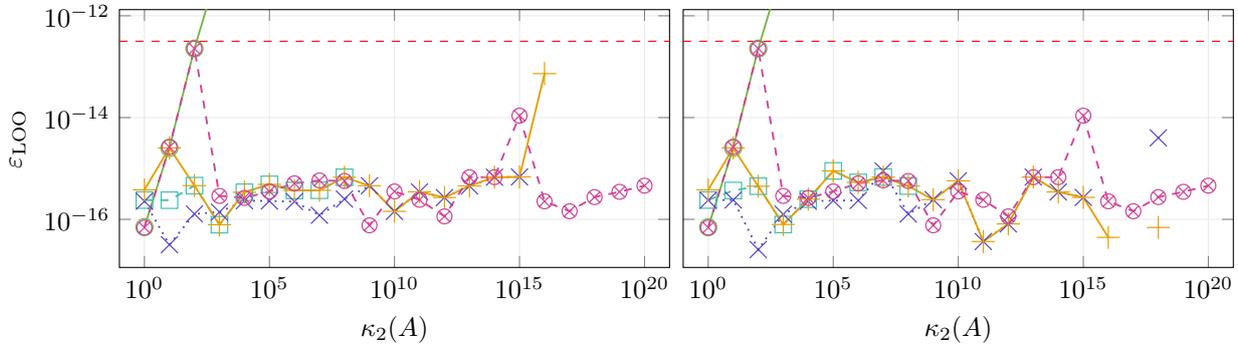
\begin{figure*}[tp]
  \begin{subfigure}{\textwidth}
    \centering
    \newcommand{\PATHA}[0]{data/1647541c-pamm_chol_qr_over_cond.csv}
\newcommand{\PATHB}[0]{data/d0a6fb85-pamm_chol_qr_over_cond.csv}

\newcommand{\addPlots}[2]
{
    \addPlotCsv{CholQR-#1}{#2}
    \addPlotCsv{CholQR2-#1}{#2}
    \addPlotCsv{sCholQR-#1}{#2}
    \addPlotCsv{isCholQR-#1}{#2}
    \addPlotCsv{rsCholQR-#1}{#2}
}

\begin{tikzpicture}%
    \pgfplotsset
    {
        width = .49\textwidth,
        height = 5cm,
        yticklabel style={anchor=east, text width=\tickwidth, align=right},
        ylabel style = {rotate=0, anchor=south, align=center, text width=5cm},
        legend cell align = left,
        minor y tick num=9,
        enlarge x limits=0.05,
        enlarge y limits=0.1,
        xmin=1e0, xmax=1e20,
        ymin=0, ymax=1.8e-3,
        xtickten={0,5,...,20},
        ClCholQR,
    }
    \begin{axis}
    [
        legend columns = 5,
        legend style = {at={(1, 1)}, anchor=south, xshift=0.005\textwidth, yshift=4em},
        xmode=log,
        ylabel={Runtime [s]}, 
        legend entries = {\cholQR,\cholQRtwo,\sCholQR,\isCholQR,\rsCholQR}, 
        name = P0, xlabel={$\cond{A}$}, 
    ]
        \addPlots{runtime}{\PATHA}
    \end{axis}
    \begin{axis}
    [
        legend columns = 4,
        legend style = {at={(0.5, 1)}, anchor=south},
        xmode=log,
        scaled y ticks=false,
        %
        name = P1, xlabel={$\cond{A}$}, 
        at={(P0.east)}, anchor=west, xshift=0.01\textwidth, yticklabel=\empty,
    ]
        \addPlots{runtime}{\PATHB}
    \end{axis}
    \node[anchor=south, yshift=1em] at (P0.north) {\bfseries Server};
    \node[anchor=south, yshift=1em] at (P1.north) {\bfseries Notebook};
\end{tikzpicture}%

    \caption{Runtime measurements. \rsCholQR{} has a higher overhead due to the {\pymor} abstraction layer. Its runtime increases for higher matrix condition numbers $\cond{A}$ due to more iterations and shifts being required to achieve orthogonality.}
    \label{fig:chol_runtime}
  \end{subfigure}
  \begin{subfigure}{\textwidth}
    \centering
    \newcommand{\PATHA}[0]{data/1647541c-pamm_chol_qr_over_cond.csv}
\newcommand{\PATHB}[0]{data/d0a6fb85-pamm_chol_qr_over_cond.csv}

\newcommand{\addPlots}[2]
{
    \addPlotCsv{CholQR-#1}{#2}
    \addPlotCsv{CholQR2-#1}{#2}
    \addPlotCsv{sCholQR-#1}{#2}
    \addPlotCsv{isCholQR-#1}{#2}
    \addPlotCsv{rsCholQR-#1}{#2}
}

\begin{tikzpicture}%
    \pgfplotsset
    {
        width = .49\textwidth,
        height = 5cm,
        yticklabel style={anchor=east, text width=\tickwidth, align=right},
        ylabel style = {rotate=0, anchor=south, align=center, text width=5cm},
        legend cell align = left,
        minor y tick num=9,
        enlarge x limits=0.05,
        enlarge y limits=0.1,
        xmin=1e0, xmax=1e20,
        ymin=3e-17, ymax=5e-13,
        xtickten={0,5,...,20},
        scaled ticks=false,
        ClCholQR,
    }
    \begin{axis}
    [
        legend columns = 5,
        legend style = {at={(1, 1)}, anchor=south, xshift=0.005\textwidth, yshift=4em},
        ymode=log,
        xmode=log,
        ylabel={\err{LOO}}, 
        name = P0, xlabel={$\cond{A}$}, 
    ]
        \addPlots{orth-loss}{\PATHA}
        \addplot [red,dashed,domain=10^(-1):10^21,samples=2]{3.16*10^(-13)};
    \end{axis}
    \begin{axis}
    [
        legend columns = 4,
        legend style = {at={(0.5, 1)}, anchor=south},
        ymode=log,
        xmode=log,
        scaled y ticks=false,
        %
        name = P1, xlabel={$\cond{A}$}, 
        at={(P0.east)}, anchor=west, xshift=0.01\textwidth, yticklabel=\empty,
    ]
        \addPlots{orth-loss}{\PATHB}
        \addplot [red,dashed,domain=10^(-1):10^21,samples=2]{3.16*10^-13};
    \end{axis}
\end{tikzpicture}%

    \caption{Loss of orthogonality $\err{LOO}(Q)=\norm{I-Q\tra Q}$.}
    \label{fig:chol_loo}
  \end{subfigure}
  \caption{Comparison of QR decomposition methods for $A\in\R[300][10]$ and varying condition numbers $\kappa_2(A)=10^x$, $x\in[0,1,\ldots,20]$, \(50\) repetitions. \cholQR{} and \cholQRtwo{} fail to compute the Cholesky decomposition for $\cond{A}>\e{8}$. The same occurs for \sCholQR{} and \isCholQR{} for $\cond{A}\gtrapprox\e{16}$. \rsCholQR{} is numerical robust for the given benchmark.}
  \label{fig:chol_figure}
\end{figure*}
\begin{figure*}[tp]
    \begin{subfigure}{\textwidth}
        \centering
    \newcommand{\PATHA}[0]{data/1647541c-pamm_main_over_cond.csv}
\newcommand{\PATHB}[0]{data/d0a6fb85-pamm_main_over_cond.csv}

\newcommand{\addPlots}[2]
{
    \addPlotCsv{ScipyQR-#1}{#2}
    \addPlotCsv{MGS2-#1}{#2}
    \addPlotCsv{rsCholQR-#1}{#2}
    \addPlotCsv{pnCholQR(2)-#1}{#2}
    \addPlotCsv{pnCholQR(3)-#1}{#2}
    \addPlotCsv{pnCholQR(4)-#1}{#2}
    \addPlotCsv{pnCholQR(5)-#1}{#2}
}

\begin{tikzpicture}%
    \pgfplotsset
    {
        width = .49\textwidth,
        height = 5cm,
        ylabel style = {rotate=0, anchor=south, align=center, text width=5cm},
        legend cell align = left,
        minor y tick num=9,
        enlarge x limits=0.05,
        enlarge y limits=0.1,
        xmin=1e0, xmax=1e20,
        ymin=1e-16, ymax=1e-13,
        xtickten={0,5,...,20},
        ClLinePlot,
    }
    \begin{axis}
    [
        legend columns = 4,
        legend style = {at={(1, 1)}, anchor=south, xshift=0.005\textwidth, yshift=4em},
        ymode=log,
        xmode=log,
        ylabel={\err{LOO}},
        legend entries = {\scipyQR,\mgs,\rsCholQR,\pnCholQR[2],\pnCholQR[3],\pnCholQR[4],\pnCholQR[5]}, 
        name = P0, xlabel={$\cond{A}$}, 
    ]
        \addPlots{orth-loss}{\PATHA}
    \end{axis}
    \begin{axis}
    [
        legend columns = 4,
        legend style = {at={(0.5, 1)}, anchor=south},
        ymode=log,
        xmode=log,
        scaled y ticks=false,
        %
        name = P1, xlabel={$\cond{A}$}, 
        at={(P0.east)}, anchor=west, xshift=0.01\textwidth, yticklabel=\empty,
    ]
        \addPlots{orth-loss}{\PATHB}
    \end{axis}
    \node[anchor=south, yshift=1em] at (P0.north) {\bfseries Server};
    \node[anchor=south, yshift=1em] at (P1.north) {\bfseries Notebook};
\end{tikzpicture}%

        \caption{Loss of orthogonality $\err{LOO}(Q)=\norm{I-Q\tra Q}$.}
        \label{fig:loo_cond}
    \end{subfigure}
    \begin{subfigure}{\textwidth}
        \centering
    \newcommand{\PATHA}[0]{data/1647541c-pamm_main_over_cond.csv}
\newcommand{\PATHB}[0]{data/d0a6fb85-pamm_main_over_cond.csv}

\newcommand{\addPlots}[2]
{
    \addPlotCsv{ScipyQR-#1}{#2}
    \addPlotCsv{MGS2-#1}{#2}
    \addPlotCsv{rsCholQR-#1}{#2}
    \addPlotCsv{pnCholQR(2)-#1}{#2}
    \addPlotCsv{pnCholQR(3)-#1}{#2}
    \addPlotCsv{pnCholQR(4)-#1}{#2}
    \addPlotCsv{pnCholQR(5)-#1}{#2}
}

\begin{tikzpicture}%
    \pgfplotsset
    {
        width = .49\textwidth,
        height = 5cm,
        ylabel style = {rotate=0, anchor=south, align=center, text width=5cm},
        legend cell align = left,
        minor y tick num=9,
        enlarge x limits=0.05,
        enlarge y limits=0.1,
        xmin=1e0, xmax=1e20,
        ymin=5e-17, ymax=2.5e-14,
        xtickten={0,5,...,20},
        ClLinePlot,
    }
    \begin{axis}
    [
        legend columns = 4,
        legend style = {at={(1, 1)}, anchor=south, xshift=0.005\textwidth, yshift=4em},
        ymode=log,
        xmode=log,
        %
        ylabel={\err{RRR}},
        name = P0, xlabel={$\cond{A}$}, 
    ]
        \addPlots{rel-rec-res}{\PATHA}
    \end{axis}
    \begin{axis}
    [
        legend columns = 4,
        legend style = {at={(0.5, 1)}, anchor=south},
        ymode=log,
        xmode=log,
        scaled y ticks=false,
        %
        name = P1, xlabel={$\cond{A}$}, 
        at={(P0.east)}, anchor=west, xshift=0.01\textwidth, yticklabel=\empty,
    ]
        \addPlots{rel-rec-res}{\PATHB}
    \end{axis}
\end{tikzpicture}%

        \caption{Relative reconstruction residual $\err{RRR}(A,Q,R)=\frac{\norm{A-QR}}{\norm{A}}$.}
        \label{fig:rrr_cond}
    \end{subfigure}
    \begin{subfigure}{\textwidth}
        \centering
    \newcommand{\PATHA}[0]{data/1647541c-pamm_main_over_cond.csv}
\newcommand{\PATHB}[0]{data/d0a6fb85-pamm_main_over_cond.csv}

\newcommand{\addPlots}[2]
{
    \addPlotCsv{ScipyQR-#1}{#2}
    \addPlotCsv{MGS2-#1}{#2}
    \addPlotCsv{rsCholQR-#1}{#2}
    \addPlotCsv{pnCholQR(2)-#1}{#2}
    \addPlotCsv{pnCholQR(3)-#1}{#2}
    \addPlotCsv{pnCholQR(4)-#1}{#2}
    \addPlotCsv{pnCholQR(5)-#1}{#2}
}

\begin{tikzpicture}%
    \pgfplotsset
    {
        width = .49\textwidth,
        height = 5cm,
        ylabel style = {rotate=0, anchor=south, align=center, text width=5cm},
        legend cell align = left,
        minor y tick num=9,
        enlarge x limits=0.05,
        enlarge y limits=0.1,
        xmin=1e0, xmax=1e20,
        ymin=5e-17, ymax=1e-14,
        xtickten={0,5,...,20},
        ClLinePlot,
    }
    \begin{axis}
    [
        legend columns = 4,
        legend style = {at={(1, 1)}, anchor=south, xshift=0.005\textwidth, yshift=4em},
        ymode=log,
        xmode=log,
        scaled y ticks=false,
        %
        ylabel={\err{RCR}},
        name = P0, xlabel={$\cond{A}$}, 
    ]
        \addPlots{rel-chol-res}{\PATHA}
    \end{axis}
    \begin{axis}
    [
        legend columns = 4,
        legend style = {at={(0.5, 1)}, anchor=south},
        ymode=log,
        xmode=log,
        %
        name = P1, xlabel={$\cond{A}$}, 
        at={(P0.east)}, anchor=west, xshift=0.01\textwidth, yticklabel=\empty,
    ]
        \addPlots{rel-chol-res}{\PATHB}
    \end{axis}
\end{tikzpicture}%

        \caption{Relative Cholesky residual $\err{RCR}(A,R)=\frac{\norm{A\tra A-R\tra R}}{\norm{A}^2}$.}
        \label{fig:rcr_cond}
    \end{subfigure}
    \label{fig:errors_cond_figure}
    \caption{QR quality measures over varying matrix condition number of a matrix $A\in\R[\e{6}][100]$ and condition numbers $\kappa_2(A)=10^x$, $x\in[0,1,\ldots,20]$, \(10\) repetitions. Due to the rank-revealing features of {\mgs}, its relative reconstruction residual $\err{RRR}$ increases for $\cond{A}\geq\e{15}$. All algorithms provide good solutions and do not fail.}
\end{figure*}
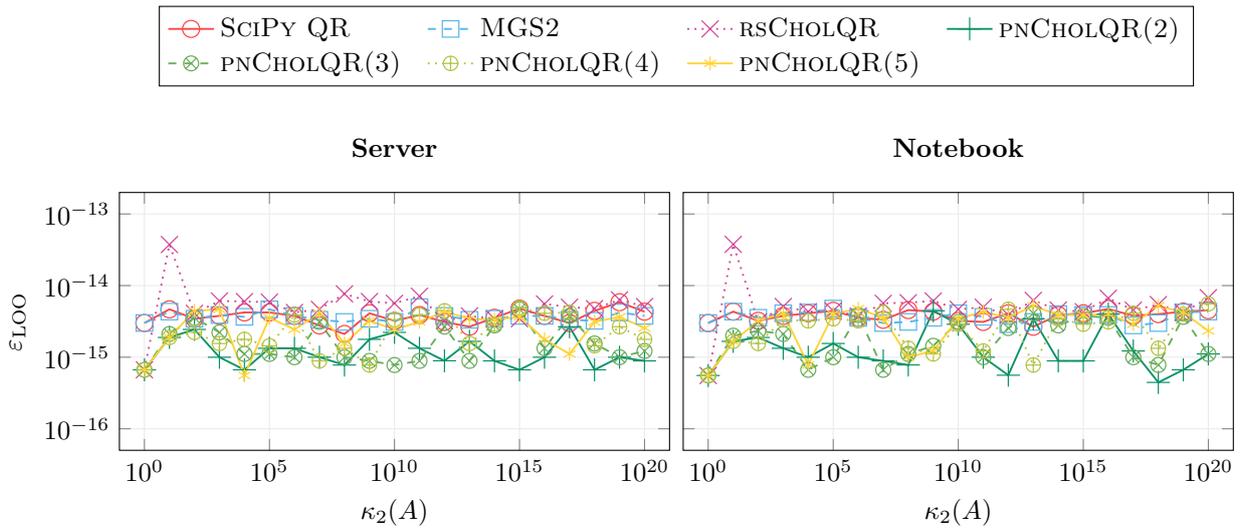
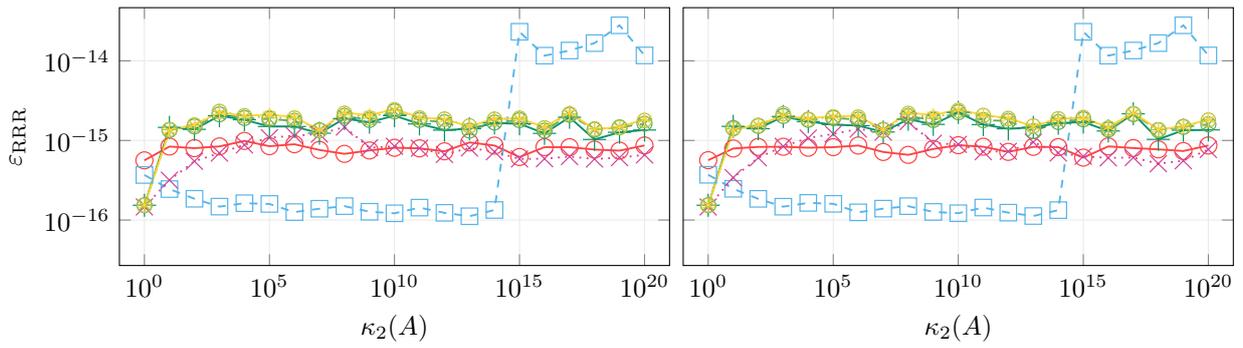
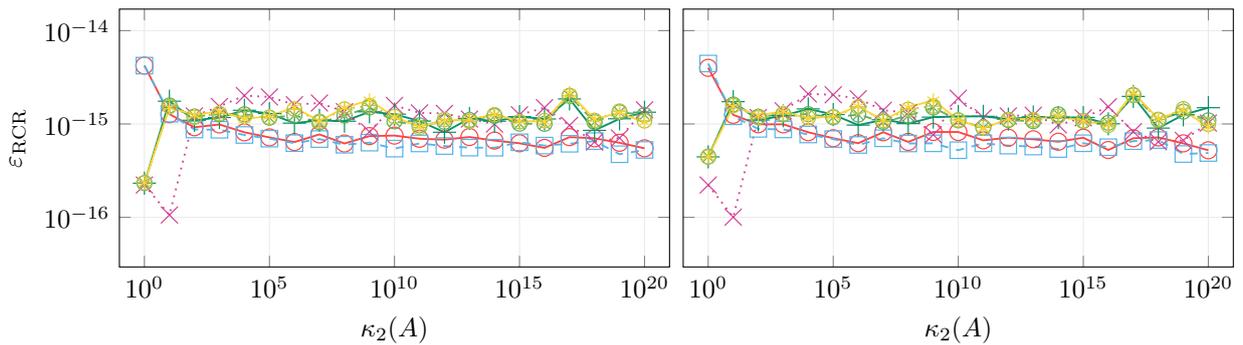
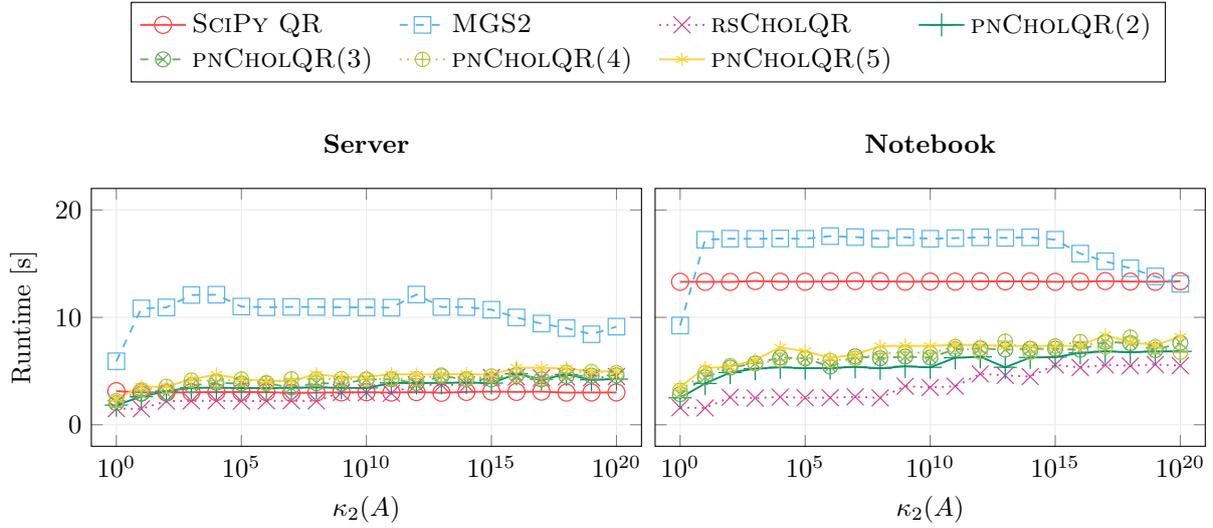
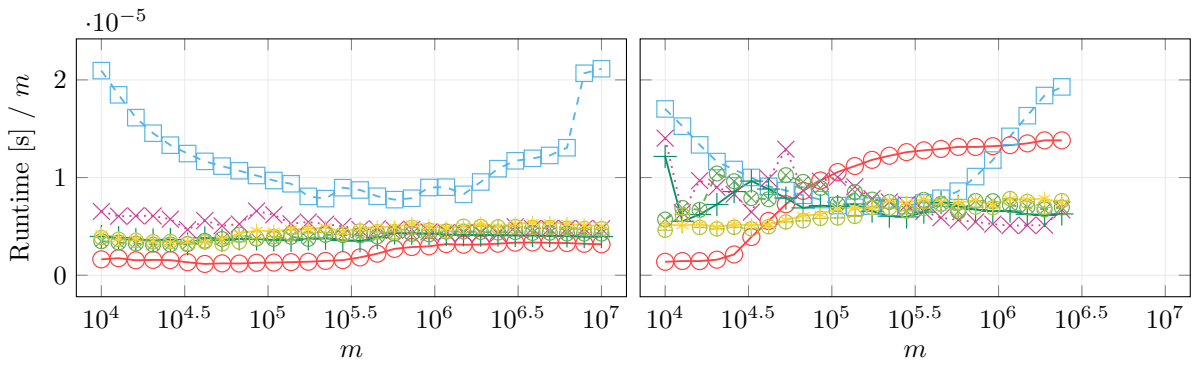
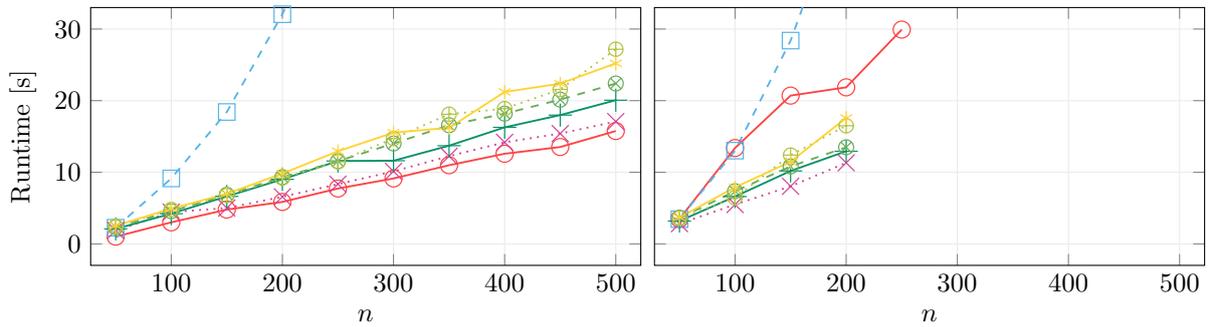
\begin{figure*}[tp]
    \begin{subfigure}{\textwidth}
        \centering
    \newcommand{\PATHA}[0]{data/1647541c-pamm_main_over_cond.csv}
\newcommand{\PATHB}[0]{data/d0a6fb85-pamm_main_over_cond.csv}

\newcommand{\addPlots}[2]
{
    \addPlotCsv{ScipyQR-#1}{#2}
    \addPlotCsv{MGS2-#1}{#2}
    \addPlotCsv{rsCholQR-#1}{#2}
    \addPlotCsv{pnCholQR(2)-#1}{#2}
    \addPlotCsv{pnCholQR(3)-#1}{#2}
    \addPlotCsv{pnCholQR(4)-#1}{#2}
    \addPlotCsv{pnCholQR(5)-#1}{#2}
}

\begin{tikzpicture}%
    \pgfplotsset
    {
        width = .49\textwidth,
        height = 5cm,
        ylabel style = {rotate=0, anchor=south, align=center, text width=5cm},
        legend cell align = left,
        minor y tick num=9,
        enlarge x limits=0.05,
        enlarge y limits=0.1,
        xmin=1e0, xmax=1e20,
        ymin=0, ymax=20,
        xtickten={0,5,...,20},
        ClLinePlot,
    }
    \begin{axis}
    [
        legend columns = 4,
        legend style = {at={(1, 1)}, anchor=south, xshift=0.005\textwidth, yshift=4em},
        xmode=log,
        ylabel={Runtime [s]},
        legend entries = {\scipyQR,\mgs,\rsCholQR,\pnCholQR[2],\pnCholQR[3],\pnCholQR[4],\pnCholQR[5]}, 
        name = P0, xlabel={$\cond{A}$}, 
    ]
        \addPlots{runtime}{\PATHA}
    \end{axis}
    \begin{axis}
    [
        legend columns = 4,
        legend style = {at={(0.5, 1)}, anchor=south},
        xmode=log,
        scaled y ticks=false,
        %
        name = P1, xlabel={$\cond{A}$}, 
        at={(P0.east)}, anchor=west, xshift=0.01\textwidth, yticklabel=\empty,
    ]
        \addPlots{runtime}{\PATHB}
    \end{axis}
    \node[anchor=south, yshift=1em] at (P0.north) {\bfseries Server};
    \node[anchor=south, yshift=1em] at (P1.north) {\bfseries Notebook};
\end{tikzpicture}%

        \caption{Runtime over varying matrix condition number \cond{A}, $A\in\R[\e{6}][100]$, $\kappa_2(A)=10^x$, $x\in[0,1,\ldots,20]$.}
        \label{fig:runtime_cond}
    \end{subfigure}
    \begin{subfigure}{\textwidth}
        \centering
    \newcommand{\PATHA}[0]{data/1647541c-pamm_main_over_m.csv}
\newcommand{\PATHB}[0]{data/d0a6fb85-pamm_main_over_m.csv}

\renewcommand{\addPlotCsv}[2]{
    \addplot table [y=#1, col sep=comma, empty line=jump, unbounded coords=jump, y expr=\thisrow{#1}/\thisrow{ID}] {#2};
}

\newcommand{\addPlots}[2]
{
    \addPlotCsv{ScipyQR-#1}{#2}
    \addPlotCsv{MGS2-#1}{#2}
    \addPlotCsv{rsCholQR-#1}{#2}
    \addPlotCsv{pnCholQR(2)-#1}{#2}
    \addPlotCsv{pnCholQR(3)-#1}{#2}
    \addPlotCsv{pnCholQR(4)-#1}{#2}
    \addPlotCsv{pnCholQR(5)-#1}{#2}
}

\begin{tikzpicture}%
    \pgfplotsset
    {
        width = .49\textwidth,
        height = 5cm,
        ylabel style = {rotate=0, anchor=south, align=center, text width=5cm},
        legend cell align = left,
        minor y tick num=9,
        enlarge x limits=0.05,
        enlarge y limits=0.1,
        xmin=1e4, xmax=1e7,
        ymin=0, ymax=2.2e-5,
        xtickten={4,4.5,...,7},
        ClLinePlot,
    }
    \begin{axis}
    [
        legend columns = 4,
        legend style = {at={(1, 1)}, anchor=south, xshift=0.005\textwidth, yshift=4em},
        xmode=log,
        ylabel={Runtime [s] / $m$},
        name = P0, xlabel={$m$}, 
    ]
        \addPlots{runtime}{\PATHA}
    \end{axis}
    \begin{axis}
    [
        legend columns = 4,
        legend style = {at={(0.5, 1)}, anchor=south},
        xmode=log,
        scaled y ticks=false,
        %
        name = P1, xlabel={$m$}, 
        at={(P0.east)}, anchor=west, xshift=0.01\textwidth, yticklabel=\empty,
    ]
        \addPlots{runtime}{\PATHB}
    \end{axis}
\end{tikzpicture}%

        \caption{Relative runtime $\frac{\mathrm{runtime}}{m}$ over varying vector length $m$, $A\in\R[m][100]$, $m\in[\e{4},\ldots,\e{7}]$, $\kappa_2(A)=10^{20}$. The notebook is unable to compute the QR decompositions for $m\geq \e{6.48}$ due to memory limitations.}
        \label{fig:runtime_m}
    \end{subfigure}
    \begin{subfigure}{\textwidth}
        \centering
    \newcommand{\PATHA}[0]{data/1647541c-pamm_main_over_n.csv}
\newcommand{\PATHB}[0]{data/d0a6fb85-pamm_main_over_n.csv}

\newcommand{\addPlots}[2]
{
    \addPlotCsv{ScipyQR-#1}{#2}
    \addPlotCsv{MGS2-#1}{#2}
    \addPlotCsv{rsCholQR-#1}{#2}
    \addPlotCsv{pnCholQR(2)-#1}{#2}
    \addPlotCsv{pnCholQR(3)-#1}{#2}
    \addPlotCsv{pnCholQR(4)-#1}{#2}
    \addPlotCsv{pnCholQR(5)-#1}{#2}
}

\begin{tikzpicture}%
    \pgfplotsset
    {
        width = .49\textwidth,
        height = 5cm,
        ylabel style = {rotate=0, anchor=south, align=center, text width=5cm},
        legend cell align = left,
        minor y tick num=9,
        enlarge x limits=0.05,
        enlarge y limits=0.1,
        xmin=50, xmax=500,
        ymin=0, ymax=30,
        ClLinePlot,
    }
    \begin{axis}
    [
        legend columns = 4,
        legend style = {at={(1, 1)}, anchor=south, xshift=0.005\textwidth, yshift=1em},
        %
        ylabel={Runtime [s]}, 
        name = P0, xlabel={$n$}, 
    ]
        \addPlots{runtime}{\PATHA}
    \end{axis}
    \begin{axis}
    [
        legend columns = 4,
        legend style = {at={(0.5, 1)}, anchor=south},
        scaled y ticks=false,
        %
        name = P1, xlabel={$n$}, 
        at={(P0.east)}, anchor=west, xshift=0.01\textwidth, yticklabel=\empty,
    ]
        \addPlots{runtime}{\PATHB}
    \end{axis}
\end{tikzpicture}%

        \caption{Runtime over varying number of columns $n$, $A\in\R[\e{6}][n]$, $n\in[50,100,\ldots,500]$, $\kappa_2(A)=10^{20}$. The notebook is unable to compute the QR decompositions for $n\geq 250$ due to memory limitations.}
        \label{fig:runtime_n}
    \end{subfigure}
    \label{fig:runtime_figure}
    \caption{Runtime comparison by variation of different aspects of the original matrix \(A\) with \(10\) repetitions.}
\end{figure*}
\begin{figure*}[tp]
  \centering
    \newcommand{\PATHA}[0]{data/1647541c-pamm_main_over_many_n.csv}

\newcommand{\addPlots}[2]
{
    \addPlotCsv{ScipyQR-#1}{#2}
    \pgfplotsset{cycle list shift=1}
    \addPlotCsv{rsCholQR-#1}{#2}
    \addPlotCsv{pnCholQR(2)-#1}{#2}
    \addPlotCsv{pnCholQR(3)-#1}{#2}
    \addPlotCsv{pnCholQR(4)-#1}{#2}
    \addPlotCsv{pnCholQR(5)-#1}{#2}
}
\pgfkeys{/pgf/number format/1000 sep={\,}}
\begin{tikzpicture}%
    \pgfplotsset
    {
        width = .95\textwidth,
        height = 6cm,
        yticklabel style={anchor=east, text width=\tickwidth, align=right},
        ylabel style = {rotate=0, anchor=south, align=center, text width=5cm},
        legend cell align = left,
        minor y tick num=9,
        enlarge x limits=0.05,
        enlarge y limits=0.1,
        xmin=500, xmax=10000,
        xtick = {1000,2000,...,10000},
        scaled ticks=false,
        ClLinePlot,
    }
    \begin{axis}
    [
        legend columns = 4,
        legend style = {at={(0.5, 1)}, anchor=south, xshift=0.005\textwidth, yshift=1em},
        %
        legend entries = {\scipyQR,\rsCholQR,\pnCholQR[2],\pnCholQR[3],\pnCholQR[4],\pnCholQR[5]}, 
        ylabel={Runtime [s]}, 
        name = P0, 
    ]
        \addPlots{runtime}{\PATHA}
    \end{axis}
\end{tikzpicture}%

  \caption{Runtime over varying number of columns $n$ with 3 repetitions, $A\in\R[\e{6}][n]$, $n\in[500,1\,000,\ldots,10\,000]$, $\kappa_2(A)=10^{5}$.}
  \label{fig:runtime_many_n}
\end{figure*}
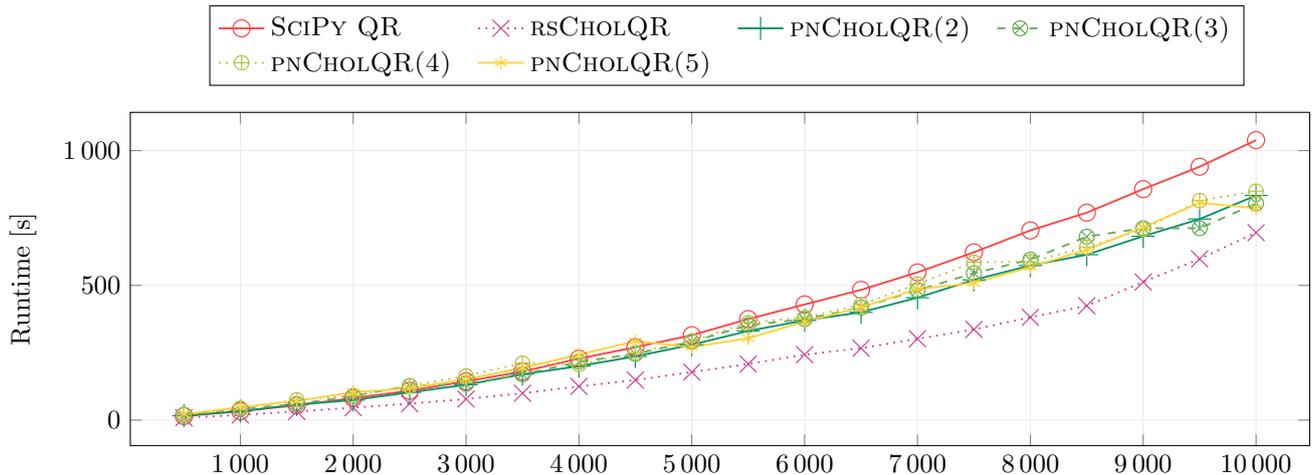
\begin{figure*}[tp]
  \begin{subfigure}{\textwidth}
    \centering
    \newcommand{\PATHA}[0]{data/16-47-pamm_main_over_va_lcond.csv}
\newcommand{\PATHB}[0]{data/d0-e8-pamm_main_over_va_lcond.csv}

\begin{tikzpicture}%
    \pgfplotsset{
        width = .95\textwidth,
        height = 5cm,
        major x tick style = transparent,
        ybar,
        ybar=2*\pgflinewidth,
        ymajorgrids = true,
        symbolic x coords={numpy, numpy-list, ngsolve-list, fenics-list},
        xtick={numpy, numpy-list, ngsolve-list, fenics-list},
        xticklabels={\texttt{Numpy-},\texttt{Numpy-},\texttt{NGSolve-}, \texttt{FEniCS-}},
        ytick={0,25,50,75,100,125}, ymin=0, ymax=105,
        ylabel style = {rotate=0, anchor=south, align=center, text width=2cm},
        scaled y ticks = false,
        enlarge x limits=0.2,
        legend columns = 4,
        legend style = {at={(0.5, 1)}, anchor=south, yshift=1em},
        legend cell align = left,
        ClBarPlot,
        every node near coord/.append style={rotate=45, anchor=south west, xshift=-4pt, execute at begin node={\scriptsize$\times$}, /pgf/number format/.cd, fixed, precision=2},
        point meta=explicit, nodes near coords, nodes near coords align={vertical}, every node near coord/.append style={ /pgf/number format/.cd, fixed, precision=2 },
    }
    \begin{axis}
        [
            name=P0,
            bar width=1em,
            ylabel = {Server Runtime [s]},
            xticklabel=\empty,
            legend entries = {\scipyQR,\mgs,\rsCholQR,\pnCholQR[2],\pnCholQR[3],\pnCholQR[4],\pnCholQR[5]}, 
        ]
        \foreach \func in {ScipyQR,MGS2,rsCholQR,pnCholQR(2),pnCholQR(3),pnCholQR(4),pnCholQR(5)}
        {
            \addplot table [y=\func-runtime, meta=\func-rel-runtime, col sep=comma, empty line=jump] {\PATHA};
        }
    \end{axis}
    \begin{axis}
        [
            name=P1,
            at={(P0.south)}, yshift=-0.01\textwidth, anchor=north,
            bar width=1em,
            ylabel = {Notebook Runtime [s]},
            xticklabel style={name=xtick\ticknum},
        ]
        \foreach \func in {ScipyQR,MGS2,rsCholQR,pnCholQR(2),pnCholQR(3),pnCholQR(4),pnCholQR(5)}
        {
            \addplot table [y=\func-runtime, meta=\func-rel-runtime, col sep=comma, empty line=jump] {\PATHB};
        }
    \end{axis}
    \path (xtick0.north) -- (xtick1.north) coordinate[midway] (V);
    \node[anchor=north, shift={(0,-.25em)}] at (xtick0.south) {\texttt{VectorArray}};
    \node[anchor=north, shift={(0,-.25em)}] (LVA) at (xtick2.south) {\texttt{ListVectorArray}};
    \draw  (V |- xtick0.north) -- (V |- LVA.south);
\end{tikzpicture}%

    \caption{$\kappa_2(A)=10^{5}$.}
    \label{fig:runtime_va_lcond}
  \end{subfigure}
  \begin{subfigure}{\textwidth}
    \centering
    \newcommand{\PATHA}[0]{data/16-47-pamm_main_over_va_hcond.csv}
\newcommand{\PATHB}[0]{data/d0-e8-pamm_main_over_va_hcond.csv}

\begin{tikzpicture}%
    \pgfplotsset{
        width = .95\textwidth,
        height = 5cm,
        major x tick style = transparent,
        ybar,
        ybar=2*\pgflinewidth,
        ymajorgrids = true,
        symbolic x coords={numpy, numpy-list, ngsolve-list, fenics-list},
        xtick={numpy, numpy-list, ngsolve-list, fenics-list},
        xticklabels={\texttt{Numpy-},\texttt{Numpy-},\texttt{NGSolve-}, \texttt{FEniCS-}},
        ytick={0,25,50,75,100,125}, ymin=0, ymax=132,
        ylabel style = {rotate=0, anchor=south, align=center, text width=2cm},
        scaled y ticks = false,
        enlarge x limits=0.2,
        legend columns = 4,
        legend style = {at={(0.5, 1)}, anchor=south, yshift=1em},
        legend cell align = left,
        ClBarPlot,
        every node near coord/.append style={rotate=45, anchor=south west, xshift=-4pt, execute at begin node={\scriptsize$\times$}, /pgf/number format/.cd, fixed, precision=2},
        point meta=explicit, nodes near coords, nodes near coords align={vertical}, every node near coord/.append style={ /pgf/number format/.cd, fixed, precision=2 },
    }
    \begin{axis}
        [
            name=P0,
            bar width=1em,
            ylabel = {Server Runtime [s]},
            xticklabel=\empty,
        ]
        \foreach \func in {ScipyQR,MGS2,rsCholQR,pnCholQR(2),pnCholQR(3),pnCholQR(4),pnCholQR(5)}
        {
            \addplot table [y=\func-runtime, meta=\func-rel-runtime, col sep=comma, empty line=jump] {\PATHA};
        }
    \end{axis}
    \begin{axis}
        [
            name=P1,
            at={(P0.south)}, yshift=-0.01\textwidth, anchor=north,
            bar width=1em,
            ylabel = {Notebook Runtime [s]},
            xticklabel style={name=xtick\ticknum},
        ]
        \foreach \func in {ScipyQR,MGS2,rsCholQR,pnCholQR(2),pnCholQR(3),pnCholQR(4),pnCholQR(5)}
        {
            \addplot table [y=\func-runtime, meta=\func-rel-runtime, col sep=comma, empty line=jump] {\PATHB};
        }
    \end{axis}
    \path (xtick0.north) -- (xtick1.north) coordinate[midway] (V);
    \node[anchor=north, shift={(0,-.25em)}] at (xtick0.south) {\texttt{VectorArray}};
    \node[anchor=north, shift={(0,-.25em)}] (LVA) at (xtick2.south) {\texttt{ListVectorArray}};
    \draw  (V |- xtick0.north) -- (V |- LVA.south);
\end{tikzpicture}%

    \caption{$\kappa_2(A)=10^{20}$.}
    \label{fig:runtime_va_hcond}
  \end{subfigure}
  \label{fig:runtime_va}
  \caption{Runtime over varying \texttt{VectorArray} backends with 10 repetitions, $A\in\R[\e{6}][100]$. {\scipyQR} is shown in the category \texttt{NumpyVectorArray}, but works on a {\ndarray}.}
\end{figure*}
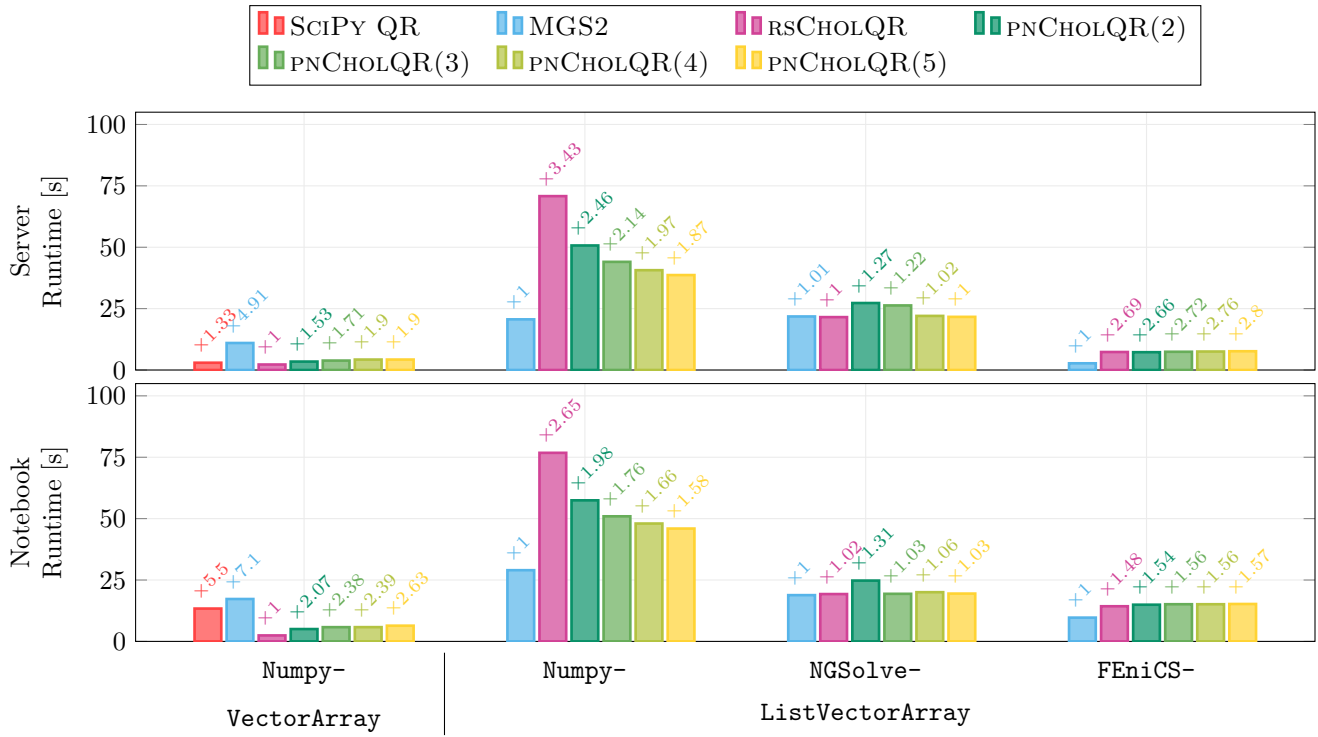
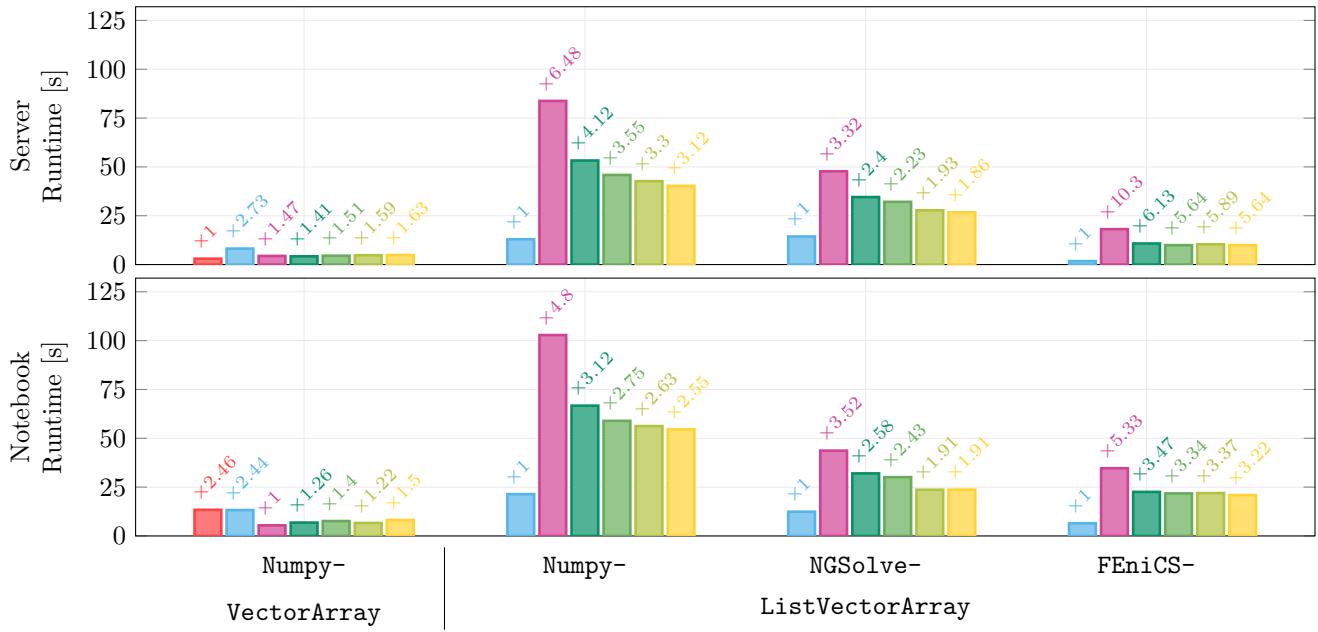
In the following we introduce our test framework and evaluate our numerical experiments. First, we specify the used computer systems, our reproducible software environments, the tested algorithms, our test matrix construction and the metrics we evaluated. Afterwards, we compare the newly introduced {\rsCholQR} to already existing Cholesky~QR variants, before illustrating the numerical robustness of {\rsCholQR} and {\pnCholQR}.
Lastly, we show our runtime measurements for a variety of aspects of the input matrix $A$ and the used \texttt{VectorArray} backend.

\subsection{Implementation and test environment}
The benchmark runs are executed on two systems. The system details for both are listed in \Cref{tab:bench_systems}. The Server system has two CPU sockets, i.e., two NUMA\footnote{\url{https://en.wikipedia.org/wiki/Non-uniform_memory_access}} nodes. For our experiments, only the first CPU was used -- via node pinning using \texttt{sched\_setaffinity} from the Linux kernel's scheduler instructions -- to ensure uniform memory access. Therefore, only up to \(64\) physical cores and \(500\)~GB RAM can be used. On both systems logical (hyperthreading) cores are allowed. We use both a server and notebook to demonstrate how the performance of the algorithms is affected by hardware capabilities.

We faced difficulties creating one single reproducible Python environment that is
able to run the algorithms for all chosen \texttt{VectorArray} backends without
performance degradation of the algorithms. Therefore, we decided to use two separate
conda~\cite{conda} environments. All shown measurements were created using the
environment Main except for the measurements with the {\fenics}
backend, i.e., the \texttt{FEniCSListVectorArray} which were taken with the \fenics{}
environment. An overview of the most important software package versions can be
found in \Cref{tab:bench_software}. The full conda environment specifications
can be found as YAML files in our {\bfseries code \& data repository on Zenodo~\cite{bindhak2025}} and can easily be installed using Miniforge~\cite{miniforge} and the \texttt{conda-forge} channel.

We are comparing the new algorithms with a number of similar algorithms from the
literature. The algorithms are all implemented in Python using the {\numpy} and
{\scipy} packages as well as {\pymor}~\cite{pymor,milk2016}. The following algorithms are used in our benchmarks:
\begin{enumerate}
\item {\cholQR} and {\cholQRtwo}~\cite{fukaya2014} represent the simple and repeated
  Cholesky~QR algorithms, respectively.\label{algo_list_first}
\item Moreover, we used \emph{`shiftedCholeskyQR3 for X = QR'}~\cite[Algorithm 4.2]{fukaya2020}\footnote{Note again, in~\cite{fukaya2020} the naming scheme for $X$ and $A$ is exactly opposite to the scheme used in this paper.} ({\sCholQR}) and {\isCholQR}, which was previously introduced in \Cref{sec:recomp}. {\sCholQR} applies a shift followed by three Cholesky~QR iterations. {\isCholQR} on the other hand applies a shift, which is evaluated just once, whenever a Cholesky decomposition fails and performs as many iterations as required to achieve orthogonality. Here, we have limited the number of iterations to \(10\).
\item  Furthermore, we used the \texttt{qr} routine from {\scipy}, more precisely
  \texttt{scipy.linalg} (\scipyQR), which is a wrapper for the {\lapack} routines {\geqrf} and {\gqr}, for comparison as an efficient and readily available function for general Python users. Note that one has to explicitly request {\scipyQR} to compute an economy size QR decomposition.
  \label{algo_list_third}
\item The algorithms \texttt{gram\_schmidt} (\mgs) and\newline
  \texttt{shifted\_chol\_qr} (\rsCholQR) are part of the {\pymor} package. Here, {\mgs} is a repeated, modified Gram-Schmidt algorithm with additional rank-revealing features. The default settings of the algorithm include a drop-tolerance of \e{-13} for linearly dependent vectors.

The \texttt{shifted\_chol\_qr} routine is an implementation of \Cref{algo:cholqr,algo:cholqrupdate}, whose interface allows switching between both algorithms.\label{algo_list_fourth}
\item Lastly, we used an implementation of \Cref{alg:panel_cqr} (\pnCholQR), which is currently not part of the official {\pymor} package.\label{algo_list_last}
\end{enumerate}
While the algorithms in items~\ref{algo_list_first}\,--\,\ref{algo_list_third} work on matrices of the datatype {\ndarray}, the latter three algorithms in items~\ref{algo_list_fourth}\,\&\,\ref{algo_list_last} are all implemented based on {\pymor}'s abstract \texttt{VectorArray} interface. In principle, a \texttt{VectorArray} is a wrapper for an array of vectors of a given solver backend. Implementations for different backends exist. In this paper, we only focus on the \texttt{NumpyVectorArray} and different types of \texttt{ListVectorArray}s. A \texttt{NumpyVectorArray} is a wrapper for a {\ndarray} and maps the \texttt{VectorArray} functions onto {\ndarray} operations. As a result, not all {\ndarray} features can be leveraged, but more block operations than for a \texttt{ListVectorArray} are available (see below). A \texttt{ListVectorArray} is, as the name suggests, simply an abstract list of vectors from a solver backend. We choose to consider vectors from {\numpy}~\cite{harris2020array}, {\fenics}~\cite{AlnaesEtal2015,LoggEtal2012}, and {\ngsolve}~\cite{ngsolve,schoberl1997} in our measurements. For clarification, a \texttt{NumpyVectorArray} allows us to use many {\blas} level-3-type routines while a \texttt{NumpyListVectorArray} can only operate vector-vector-wise, i.e., only {\blas} level-1-type routines. The latter also applies to all other \texttt{ListVectorArray}s.

It is important to note that our implementations of {\rsCholQR} and {\pnCholQR}
utilize \texttt{VectorArray}s for the matrices associated with $A$ and $Q$, only. All other matrices, i.e., $R, B, X$ and $I$, are assumed to be small enough and stored as {\ndarray}s for which we can use {\blas} and {\lapack} routines.

Furthermore, algorithms {\isCholQR}, {\rsCholQR} and {\pnCholQR} all have a stopping criterion of \[\fnorm{I-Q\tra Q}\leq\e{-13},\] and we allow a maximum of \(10\) iterations (outer iterations for the {\pnCholQR}). Also, all employed algorithms work on a copy of the input matrix rather than in-place, since we are not able to exclude the copy time of the {\scipyQR} function.

We generate our test matrices $A$ pseudo-randomly via a commonly found SVD approach ($A:=U\Sigma V$). Hereby, we construct two random orthogonal matrices $U\in\R[m][n], V\in\R[n][n]$ and an ill-conditioned diagonal matrix $\Sigma\in\R[n][n]$. The values in $\Sigma$ are log-equidistant starting from \(1\) up to the desired matrix condition number $\cond{A}=\frac{\sigma_\mathrm{max}}{\sigma_\mathrm{min}}$. The seed used for the pseudo-random number generator is varied with the matrix dimensions and the condition number. The test matrix and all intermediate matrices are of data type IEEE-754 double precision, i.e., \(64\)-bit floating point numbers, hence \(\uro\approx\e{-16}[1.11]\).

The following metrics are measured and evaluated:
\begin{description}
    \item[Runtime] We measure the wall-time of a QR decomposition call in seconds and evaluate the median of multiple repetitions. The number of repetitions is specified in the caption of each figure.
    \item[Loss of orthogonality]
      \[\displaystyle\err{LOO}(Q)=\norm{I-Q\tra Q}\]
    \item[Relative reconstruction residual]
      \[\displaystyle\err{RRR}(A,Q,R)=\frac{\norm{A-QR}}{\norm{A}}\]
    \item[Relative Cholesky residual]
      \[\displaystyle\err{RCR}(A,R)=\frac{\norm{A\tra A-R\tra R}}{\norm{A}^2}\]
\end{description}

\subsection{An overview of Cholesky~QR variants}%
\label{cha:bench_chol}
In this subsection, we perform a comparison between existing Cholesky~QR variants and {\rsCholQR}. We use small matrices of size $A\in\R[300][10]$ with ever-increasing matrix condition numbers, mimicking the numerical experiment in \cite[Fig. 6.4]{fukaya2020}. This is done in order to reduce theoretical restrictions on the shift. Prior to conducting the comparison, it is necessary to consider the following aspects and assumptions:
\begin{enumerate}
    \item Since {\rsCholQR} is implemented in {\pymor}, it has a higher overhead due to the additional function layer of the \texttt{VectorArray} interface. In order to implement the algorithm, our current implementation additionally requires many delete/copy calls instead of truly working in-place, which degrades the performance further.
    \item Given the norm equivalence \[\fnorm{Q\tra Q-I_n}\leq\norm{Q\tra Q-I_n}\leq \sqrt{n}\fnorm{Q\tra Q-I_n}\] between the spectral and Frobenius norm and the stopping criterion $\fnorm{Q\tra Q-I_n}\leq \e{-13}$ for {\isCholQR} and {\rsCholQR}. The stopping criterion in spectral norm is bounded from above by \[\norm{Q\tra Q-I_n}\leq \sqrt{n}\fnorm{Q\tra Q-I_n}\approx\e{-13}[3.16],\] which is represented by the red, dashed line in \Cref{fig:chol_loo}.
    \item Numerical results depend on the used hardware and even a use of virtual environments cannot
      entirely remedy this fact. This phenomenon can be
      observed by different failing behavior of {\sCholQR} and {\isCholQR} on
      the server and notebook for $\cond{A}\geq\e{-16}$. In our case, it might be
      caused by a different execution order of operations, i.e., error
      perturbation; e.g., due to a different tiling scheme of the
      underlying {\openblas} caused by varying cache sizes or the number of cores available.
\end{enumerate}

A major factor for the success, i.e., a low \err{LOO}, of Cholesky~QR algorithms is
the matrix condition number of the input matrix $A$, since it is effectively
squared for the Gramian $X=A\tra A$. As a result, the simple {\cholQR} has high
errors and {\cholQRtwo} fails to compute the Cholesky decomposition usually for
$\cond{A}\geq\sqrt{\uro\inverse}$. As can be seen, {\sCholQR} and {\isCholQR} are numerically robust for $\cond{A}\leq\e{15}$. For even higher condition numbers, they become unreliable and might fail\footnote{Here, it strongly depends on the seed when {\sCholQR} and {\isCholQR} fails. We have noticed that for the given test matrix construction and size both of them struggle for $\cond{A}\geq\e{16}$}. {\rsCholQR}, on the other hand, is able to compute the QR decomposition for matrices with such high condition numbers and that even for larger matrices, as we will see in \Cref{cha:bench_runtimes,cha:bench_errors}. However, this comes at the cost of runtime, since the norm of the Gramian has to be computed potentially multiple times. Again, we would like to point out that the runtime is negatively influenced by the \texttt{VectorArray} interface.
\subsection{Analysis of numerical errors}%
\label{cha:bench_errors}
In \Cref{fig:chol_figure}, we illustrate the numerically robustness of {\rsCholQR} and {\pnCholQR} ($r=2,\ldots,5$). Here, we only show our measurements based on a varying matrix condition number. More quality measures for the variation of different aspects of the original matrix can be found in our repository~\cite{bindhak2025}.

Due to the essence of the Cholesky~QR algorithms ($Q:=AR\inverse$), we can
expect low relative reconstruction residual errors \err{RRR}. Nonetheless, the orthogonality of \(Q\) is fully dependent on the quality of $R$, which explains the importance of including the relative Cholesky residuals \err{RCR} in our measurements. The errors of the Cholesky~QR variants are comparable to {\scipyQR} errors.

The increase in \err{RRR} for {\mgs} for high condition numbers is due to its
rank-revealing features, which are active by default. Some vectors are linearly dependent, which causes them to be dropped. Therefore, the original
matrix $A$ can only be reconstructed with higher errors. At the same time, this is also the cause for its slight speedup in \Cref{fig:runtime_cond}.
\subsection{Runtime measurements}%
\label{cha:bench_runtimes}
In the following we discuss our runtime measurements and illustrate the dependence on a multitude of factors. We investigate the changes in runtime depending on matrix properties like the dimension $\VA\in\R[m][n]$, but also the matrix condition number $\cond{A}$. Furthermore, as already mentioned, we use two different hardware systems to observe potential differences in the behavior of the algorithms. Also, since this work is done for the {\pymor} context, we look into different \texttt{VectorArray}s. Due to the limited size of the notebook's main memory, some measurements in \Cref{fig:runtime_m,fig:runtime_n} terminate for $m\geq \e{6.48}$ and $n\geq 250$, respectively.

The main takeaway of \Cref{fig:runtime_cond} is that, in contrast to {\scipyQR}, the runtime of the algorithms does not only depend on the matrix dimensions, but depends on the matrix condition number. The runtime of the Cholesky~QR variants ({\rsCholQR} and {\pnCholQR}) increases for higher condition numbers, since more iterations and shift recomputations have to be performed in order to achieve orthogonality. This is important to note, since we use $\cond{A}=\e{20}$ in \Cref{fig:runtime_m,fig:runtime_n,fig:runtime_va_hcond}. The improved robustness of {\rsCholQR} is paid by a factor of $3$ and $3.58$ in runtime for high condition numbers, compared to low condition numbers on the server and notebook, respectively. Furthermore, {\pnCholQR} has a much steeper increase in runtime for $\cond{A}\leq\e{4}$, but only slowly increases thereafter. While on the server hardware all competitors are significantly faster than {\pymor}'s default (\mgs) and of rather similar performance among each other, on the notebook the Cholesky~QR variants can even beat the computation times of {\scipyQR} by a margin, with our newly suggested {\rsCholQR} clearly fastest for all condition numbers.

In \Cref{fig:runtime_m}, we consider a constant number of vectors ($n=100$), but vary its length $m$. For better visibility, we depict the runtime in relation to $m$. Roughly speaking, QR decompositions are of the time complexity class $\Ooff{mn^2}$. The resulting line plot can be interpreted as the effective constant of the term $mn^2$. As one can see, this constant varies depending on $m$ (the tall matrix direction). In comparison to the experiments above, {\scipyQR} faces more memory management and movement. In particular one can observe this for $m\geq\e{4.5}$ on the notebook.

\Cref{fig:runtime_n} illustrates the runtime over a varying number of vectors $n$. Here, {\pnCholQR} scales consecutively worse for an increasing amount of panels $r$, where {\rsCholQR}, i.e., $r=1$, has the lowest runtime out of the Cholesky variants. {\mgs}, being a Gram-Schmidt based algorithm, has a significant increase in runtime for larger $n$. Lastly, the runtime of {\scipyQR} varies the most between the used hardware. On the server it has consecutively the lowest runtime, but on the notebook it performs even worse than {\pnCholQR[5]}. However, it is important to note that the Gramian has an insignificant size in all measurements and fits into the caches, which makes it cheap to work with.

In \Cref{fig:runtime_many_n} we repeat the test with a lower condition number $\cond{A}=\e{5}$, but up to $n=10\,000$ vectors. Here, we can observe that up to $n\leq3\,500$ the panel variants are similar in runtime to {\scipyQR}. {\rsCholQR}, on the other hand, performs much better compared to {\pnCholQR}, due to the reduced condition number (c.f. \Cref{fig:runtime_cond}). Starting with $n\geq 4\,000$ vectors the panel variants are faster than {\scipyQR} and keep a lower increase rate. Note, that the panel variants have an uneven increase in runtime for an increasing amount of vectors. Furthermore, the runtime even decreases in some cases, e.g., {\pnCholQR[5]} for $n=5\,000$ compared to $n=4\,500$. Note, many numerical solvers, like {\openblas}, use tiling schemes with powers of two for matrices. Therefore, we suspect, that the size of the Gramians is more suited. {\pnCholQR} can be derived to use panels of said sizes. Starting with $n\geq 9\,000$ vectors, the runtime of {\rsCholQR} increases at a considerably higher rate, which might be caused by the amount of data of the Gramian. For more than $n\geq 10\,000$ vectors, its runtime might even fall above the runtime of {\pnCholQR}. The L3 caches of the used server have a total size of \(256\)~MB, which means that they are unable to store the full Gramian for $n\geq5\,657$, making it evermore expensive to compute its norm and Cholesky decomposition. However, this is not as visible, since a lower matrix condition number is being used.

In \Cref{fig:runtime_va_lcond,fig:runtime_va_hcond}, we looked into the runtime
behavior depending on the \texttt{VectorArray} backend. {\scipyQR} is shown
in the category \texttt{NumpyVectorArray} but works on a {\ndarray}. Since
the {\pymor} algorithms can only work vector-wise with
\texttt{ListVectorArray}s, we can see a considerable increase in runtime for the
Cholesky~QR variants. This is due to the overhead in the Gramian construction,
which could previously be computed by use of {\gemm}, i.e., in level-3-type \blas{},
and now falls back to level-1-type \blas{}. However, here, the panel variants can show
their benefits. For the \texttt{NumpyListVectorArray} and
\texttt{NGSolveListVectorArray}, the expected speedup of two (see
\Cref{cha:flopCount}) for an increasing number of panels is approached. Nevertheless, the
runtime is larger compared to {\mgs}. The \texttt{FEniCSListVectorArray}
implementation, on the other hand, behaves notably different, where more than
two panels do not provide a meaningful speedup. However, {\mgs} does in fact
gain a substantial speedup when using {\fenics}. Note however, that the runtime
advantage of {\mgs} shrinks drastically for moderate condition numbers
(\Cref{fig:runtime_va_lcond}) and is most present when many vectors turn out to
be linearly dependent ($\kappa_2(A)=10^{20}\gg \frac{1}{\uro}$ in \Cref{fig:runtime_va_hcond}).
%
%

\section{Conclusions and Outlook}%
\label{cha:outlook}

Our proposed {\rsCholQR} and {\pnCholQR} are numerically robust even for high matrix
condition numbers (even $\kappa_2(A)=10^{20}\gg \frac{1}{\uro}$), making them an improvement compared to {\sCholQR} and
{\isCholQR}, which are working well for $\kappa_2(A)<\frac{1}{\uro}$. Additionally, depending on hardware and matrix properties
(condition number and dimensions), they can be faster than {\scipyQR}, and a similar
advantage is also seen with {\sCholQR} and {\isCholQR}. Even when using the
{\pymor} interface, lower runtimes have been observed, suggesting that dedicated
reimplementations of these algorithms may be of interest. Since our
suggestions seem to be most robust and do not lose too much performance
compared to the others, in situations where the condition number is unknown we
recommend their use. Whenever expected condition numbers are below
\(\frac{1}{\uro}\), {\sCholQR} and {\isCholQR} are still the better choices, and
even {\cholQR} may be a excellent (and fast) solution for $\kappa_2(A)<\frac{1}{\sqrt{\uro}}$.

However, there are still open questions that may help to improve our algorithmic
variants. Regarding the optimal number of panels
for {\pnCholQR}, or the best panel width, respectively. Both depend on the
number of vectors \(n\), additional research is needed.
As mentioned before, using a width of a  power of two could be beneficial for the used numerical solvers.
In theory {\pnCholQR} reduces the required flops by up to one half compared to
{\rsCholQR}, however memory access patterns require closer investigation for the runtime
optimization.
Better shifting strategies and handling of linear dependencies are also areas for improvement,
as they could potentially reduce the number of inner iterations.

Finally, it is worth noting that for the usage with ill-conditioned \texttt{VectorArray}s in
{\pymor}, our Cholesky~QR variant is currently only suitable, when used in
combination with the
\texttt{NumpyVectorArray}. Otherwise, the default rank-revealing
features of the current repeated modified Gram-Schmidt implementation are presenting an advantage that our current
formulations do not have. The use of rank-revealing computations around the
Gramian matrices is necessary for which the idea of pivoted Cholesky~QR, as
proposed in~\cite{fukaya2024}, is a promising approach.
%

\section*{Acknowledgement}
  The work of Art J. R. Pelling was funded by the Deutsche Forschungsgemeinschaft (DFG, German Research Foundation), project number \href{https://gepris.dfg.de/gepris/projekt/504367810?language=en}{504367810}.
  The work of Maximilian Bindhak has been supported by MaRDI, funded by the Deutsche Forschungsgemeinschaft (DFG, German Research Foundation), project number \href{https://gepris.dfg.de/gepris/projekt/460135501?language=en}{460135501}, NFDI 29/1 “MaRDI – Mathematische Forschungsdateninitiative”.

\addcontentsline{toc}{section}{References}
\bibliographystyle{siamplain}
\bibliography{refs}
\appendix%
\begin{table*}[tp]
    \rowcolors{2}{white}{gray!25!white}
    \begin{tabular}{|l|l|l|}\hline
        Operation     & Routine                                 & Number Operations                                 \\\hline
        $B\cdot C$    & \gemm                                   & $2ijk$                                            \\
        $B\cdot R$    & {\trmm} (Mult. with $R$ from the right) & $ij^2$                                            \\
        $R\cdot C$    & {\trmm} (Mult. with $R$ from the left)  & $kj^2$                                            \\
        $\chol(A)$    & \potrf                                  & $\frac{1}{3}n^3 + \frac{1}{2}n^2 + \frac{1}{6}n$  \\
        $\inv(R)$     & \trtri                                  & $\frac{1}{3}n^3 + \frac{2}{3}n$                   \\
        $\norm{X}$    & \texttt{scipy.linalg.eigsh}             & $\Ooff{n^2}$                                      \\
        $\fnorm{A}$   & \texttt{scipy.linalg.norm(X, ord='fro')}& $2n^2+n$                                          \\\hline
    \end{tabular}
    \caption{FLOP count of used routines.}%
    \label{tab:routinesFlop}
\end{table*}
\begin{table*}[tp]
    \rowcolors{2}{white}{gray!25!white}
    \begin{tabular}{|l|l|l|}\hline
        Line    & Routine                       & Number Operations                                                         \\\hline
        3       & {\gemm}                       & $2mn^2$                                                                   \\
        4       & \texttt{scipy.linalg.norm}    & $\left(x+1\right)\cdot\left[2n^2+n\right]$                                \\
        5       & {\potrf}                      & $x\cdot\left[\frac{1}{3}n^3 + \frac{1}{2}n^2 + \frac{1}{6}n\right]$       \\
        6       & \texttt{scipy.linalg.eigsh}   & $x\cdot\Ooff{n^2}$                                                        \\
        7       & {\potrf}                      & $x\cdot\left[\frac{1}{3}n^3 + \frac{1}{2}n^2 + \frac{1}{6}n + n\right]$   \\
        9       & {\trtri} + {\gemm}            & $x\cdot\left[\frac{1}{3}n^3 + \frac{2}{3}n + 2mn^2\right]$                \\
        10      & {\trmm}                       & $x\cdot n^3$                                                              \\
        11      & {\gemm}                       & $x\cdot 2mn^2$                                                            \\\hline
        Total:  &                               & $=2 m n^{2} + 2 n^{2} + n$                                                \\
                &                               & $+x\cdot\left[4 m n^{2} + 2 n^{3} + 3 n^{2} + \Ooff{n^2} + 3 n\right]$    \\\hline
    \end{tabular}
    \caption{Operation count for {\rsCholQR} as defined in
      \Cref{algo:cholqr}. Under the assumption that the outer loop iterates $x$
      times (condition is evaluated $x+1$ times).}%
    \label{tab:rsCholQRFlop}
\end{table*}
\begin{table*}[tp]
    \rowcolors{2}{white}{gray!25!white}
    \begin{tabular}{|l|l|l|}\hline
        Line    & Routine                       & Number Operations                                                                                                                 \\\hline
        4       & {\gemm}                       & $2 m q p$                                                                                                                         \\
        5       & 2 {\gemm}                     & $2 m p^{2} + 2 q p^{2} + p^{2}$                                                                                                   \\
        6       & \texttt{scipy.linalg.norm}    & $\left(x + 1\right)\cdot \left[2 p^{2} + p\right]$                                                                                \\
        8       & {\potrf}                      & $x\cdot \left(y + 1\right) \left[\frac{1}{3}p^{3} + \frac{1}{2}p^{2} + \frac{7}{6}p\right]$                                       \\
        9       & \texttt{scipy.linalg.eigsh}   & $x\cdot \left[\Ooff{p^{2}} + y - 1\right]$                                                                                        \\
        12      & 2 {\gemm} + {\trtri}          & $x\cdot \left[2 m p^{2} + 2 m q p + m p + \frac{1}{3}p^{3} + \frac{2}{3}p\right]$                                                 \\
        13      & {\gemm}                       & $x\cdot \left[2 q p^{2} + q p\right]$                                                                                             \\
        14      & {\trmm}                       & $x\cdot p^{3}$                                                                                                                    \\
        15      & {\gemm}                       & $x\cdot 2 m q p$                                                                                                                  \\
        16      & 2 {\gemm}                     & $x\cdot \left[2 m p^{2} + 2 q p^{2} + p^{2}\right]$                                                                               \\\hline
        Total:  &                               & $=2mp^{2} + 2mqp + 2qp^{2} + 3p^{2} + p$                                                                                          \\
                &                               & $+x\cdot\left[4mp^{2} + 4mqp + mp + \frac{5}{3}p^{3} + 4qp^{2} + \frac{7}{2}p^{2} + \Ooff{p^{2}} + qp + \frac{17}{6}p - 1\right]$ \\
                &                               & $+\left(y_1+\ldots+y_x\right)\cdot\left[\frac{1}{3}p^{3} + \frac{1}{2}p^{2} + \frac{7}{6}p + 1\right]$                            \\\hline
    \end{tabular}
    \caption{Operation count for {\cholQRUpdate} as defined in
      \Cref{algo:cholqrupdate}. Under the assumption that the outer loop
      iterates $x$ times (condition is evaluated $x+1$ times) and that the inner
      loop iterates $y_1,y_2,\ldots,y_x$ times (condition is evaluated $y+1$
      times).}%
    \label{tab:cholqrupdateFlop}
\end{table*}
\begin{table*}[tp]
    \rowcolors{2}{white}{gray!25!white}
    \begin{tabular}{|l|l|}\hline
                & Number Operations\\\hline
        Total:  & $=mn^{2} + \frac{mn^{2}}{r} + \frac{n^{3}}{r} - \frac{n^{3}}{r^{2}} + \frac{3n^{2}}{r} + n$\\
                & $+x\cdot\left[2mn^{2} + \frac{2mn^{2}}{r} + mn + \frac{2n^{3}}{r} - \frac{n^{3}}{3 r^{2}} + \frac{n^{2}}{2} + \frac{3n^{2}}{r} + \Ooff{\frac{n^{2}}{r}} + \frac{17n}{6} - r\right]$\\
                & $+xy\cdot\left[\frac{n^{3}}{3r^{2}} + \frac{n^{2}}{2 r} + \frac{7n}{6} + r\right]$\\\hline
    \end{tabular}
    \caption{Operation count for {\pnCholQR} as defined in
      \Cref{alg:panel_cqr}. Given $A\in\R[m][n]$, $0<r\leq n$ such that
      $\frac{n}{r}=p$, $r,p\in\mathbb{N}_{>0}$. Under the assumption that all
      outer and inner loops of a {\cholQRUpdate} call require $x$ and $y$
      iterations respectively. The assumption onto $x$ and $y$ is very
      pessimistic. In practice, the first few panels might require none or only
      a few repetitions. The last few panels, on the other hand, require many
      more repetitions. Furthermore, in terms of flops a failed Cholesky
      decomposition is counted as a full decomposition.}%
    \label{tab:panel_cqr}
\end{table*}
\begin{table*}[tp]
    \rowcolors{2}{white}{gray!25!white}
    \begin{tabular}{|l|l|}\hline
                                    & Number Operations                                                                                     \\\hline
        \pnCholQR                   & $mn^{2} + \frac{1}{r}mn^{2} + x\cdot \left[2mn^{2} + \frac{2}{r}mn^{2} + mn\right]$                   \\
        \rsCholQR                   & $2mn^{2} + x\cdot 4mn^{2}$                                                                            \\\hline
        $F(r)=$ {\pnCholQR} $/$ {\rsCholQR} & $\frac{1}{2} + \frac{1}{2r} + \frac{x}{4nx + 2n}$                                             \\\hline
    \end{tabular}
    \caption{Operation count ratio between {\rsCholQR} and {\pnCholQR}.}%
    \label{tab:rs_pn_flop_comp}
\end{table*}
\section{Number of Floating Point Operations Evaluation}%
\label{cha:flopCount}

In the following section we want to provide an overview of how many floating point operations an operation or routine requires. Only the number of additions and multiplications are listed. In the following we use {\blas} and {\lapack} style routine names for some operations. The routines are to represent the cost of the related operation. The flop count of these routines is documented in {\lapack}~Working~Note~41~\cite{lawn41}.
In \Cref{tab:routinesFlop} the following matrix definitions are being used: $A\in\R[n][n]$, $B\in\R[i][j]$, $C\in\R[j][k]$, upper triangular $R\in\R[j][j]$, symmetric $X\in\R[n][n]$.

Due to the limitations of the {\pymor} interface, it is not possible to evaluate $X:=Q\tra Q$ using the {\blas} routine \texttt{x(SY/HE)RK}. Furthermore, we also cannot evaluate $Q:=AR\inverse$ using {\trmm} for the multiplication with $R\inverse$ or \texttt{xTRSM} for directly solving the equation. In our implementations, for both cases, we use {\gemm} indirectly. We evaluated the following flop count based on {\gemm}.

In {\rsCholQR}, {\cholQRUpdate} and {\pnCholQR}, we compute the spectral norm of
the shift \[11(mn+n(n+1))\uro\norm{X}\] by use of the symmetric eigenvalue solver
routine \texttt{eigsh}, from \texttt{scipy.linalg}, which uses the Implicitly Restarted Lanczos
Method (IRLM) from {\arpackng}~\cite{arpack-ng} internally.
\begin{lstlisting}[caption={Used paramters for the \texttt{eigsh} function call.}, label={lst:eigsh},numbers=none]
norm = scipy.sparse.linalg.eigsh(
    X,
    k=1,
    tol=1e-2,
    return_eigenvectors=False,
    v0=np.ones([n])
)[0]
\end{lstlisting}
The used parameters
for the function call are documented in \Cref{lst:eigsh}. For the Cholesky QR we are only interested in the largest eigenvalue, which is why we set \texttt{k=1}. Furthermore, we only need a rough approximation of the eigenvalue, so setting \texttt{tol=1e-2} reduces the runtime of the routine.

Chapter 2.3.6 of the {\arpack} manual \cite{lehoucq98} states that approximately $(\textrm{ncv} - \textrm{nev})\cdot\text{cost matrix-vector-product}=(\textrm{ncv} - \textrm{nev})\cdot2n^2$ plus an additional $4n\cdot\textrm{ncv}(\textrm{ncv} - \textrm{nev})$ flops are required for the IRLM method. Here, $\textrm{ncv}$ and $\textrm{nev}$ mean, respectively, the number of Lanczos basis vectors that are being used through the course of the computation and the number of eigenvalues that are to be computed. Looking into the {\scipy} source code one can find that $\textrm{ncv}$ is evaluated to be $\textrm{ncv} = \min(n,\max(2k + 1, 20)) = \min(n,20)$. Therefore, approximately $38n^{2} + 1520n=\Ooff{n^2}$ flops are required to compute the largest eigenvalue.

\subsection{Number of flops comparison between {\rsCholQR} and {\pnCholQR}}%
\label{cha:rs_pn_flop_comp}
Given $A\in\R[m][n]$, $0<r\leq n$ such that $\frac{n}{r}=p$, $r,p\in\mathbb{N}_{>0}$. Under the assumption that all outer and inner loops of a {\cholQRUpdate} call require $x$ and $y$ iterations respectively. For simplification, we consider the case where $n^3 \ll m$, which means that  only terms with the variable $m$ are relevant for the comparison.
\begin{align}
    G(r) &= \lim_{n\rightarrow\infty} F(r)\quad(0\leq x<\infty)\nonumber\\
    &= \frac{1}{2} + \frac{1}{2r}\label{eq:Fr}
\end{align}
In \Cref{eq:Fr} one can see, that for large enough $n$ the term $\frac{x}{4nx + 2n}$ becomes insignificant. Furthermore, the function $G(r)$ approaches $\frac{1}{2}$ asymptotically for an increasing amount of panels, i.e., {\pnCholQR[n]} requires half as many flops as {\rsCholQR}. Intuitively, {\pnCholQR[1]} require as many flops as {\rsCholQR}.
%

\ifthenelse{\value{todoCounter}=0}{%
}{%
  \onecolumn{\hypersetup{pdfborder={0 0 0}}\listoftodos{}}
}
\end{document}